%% file: ex_article.tex
\documentclass[hidelinks,onefignum,onetabnum]{siamart250211}
\usepackage{todonotes}

\input{ex_shared}

\headers{Solving LQ Stochastic Control Problems with Signatures}{A. Aqsha, P. Bank, and L. S\'{a}nchez-Betancourt}

\title{
Solving Linear-Quadratic Stochastic Control Problems with Signatures
\thanks{
\textbf{Funding: }We are grateful for support from Berlin-Oxford IRTG 2544 and the Oxford-Man Institute of Quantitative Finance. 
AA's research is supported by  the Oxford-Man Institute of Quantitative Finance through the EPSRC Centre for Doctoral Training in Mathematics of Random Systems: Analysis, Modelling and Simulation (EPSRC Grant EP/S023925/1). PB gratefully acknowledges funding by the Deutsche Forschungsgemeinschaft (DFG, German Research Foundation) – CRC/TRR 388 ``Rough Analysis, Stochastic Dynamics and Related Fields'' – Project ID 516748464.
For the purpose of open access, the authors have applied a CC BY public copyright licence to any author accepted manuscript version arising from this submission.  
\textbf{Acknowledgments: }We are grateful to seminar participants at the Oxford-Princeton and Oxford-Berlin workshop for their comments and suggestions.
}}

\author{Alif Aqsha\thanks{Oxford-Man Institute of Quantitative Finance \& Mathematical Institute, University of Oxford
  (\email{alif.aqsha@maths.ox.ac.uk}, \email{leandro.sbetancourt@maths.ox.ac.uk}).}
\and Peter Bank\thanks{Institut für Mathematik, Technische Universität Berlin 
  (\email{bank@tu-berlin.de}).}
\and Leandro S\'{a}nchez-Betancourt\footnotemark[2]}

\begin{document}

\maketitle

\begin{abstract}
We study a signature-driven numerical scheme to solve  multi-dimensional linear-quadratic (LQ) stochastic control problems. Using that linear signature functionals are dense in the natural class of admissible controls, we show that our approach turns the original LQ problem into a deterministic convex quadratic polynomial optimisation. To underpin a numerical approach based on truncated signatures, we prove that the problem's value function can be approximated by finite-dimensional polynomial approximations when the truncation levels are chosen sufficiently high. Remarkably, our numerical experiments show very decent accuracy already for small truncation levels. Key tools for our analysis are (i) the algebraic representation of controlled stochastic differential equations and the associated cost function as linear functionals of the path signatures of the driving noise, (ii) the convergence of the truncated linear functionals, and (iii) the density of signature controls. 
\end{abstract}



\section{Introduction}

Stochastic optimal control is classically approached by characterising the associated value function. In sufficiently regular Markovian settings, the dynamic programming principle leads to a Hamilton--Jacobi--Bellman (HJB) partial differential equation (PDE) \cite{bensoussan2018estimation, pham2009continuous}, alternatively,  the Pontryagin maximum principle yields characterisations in terms of backward stochastic differential equations \cite{yong1999stochastic,carmona2016lectures}. These routes are powerful but quickly become challenging in high dimensions,  path-dependent settings, or with lack of regularity. Although there are results for optimal linear-quadratic (LQ) control problems driven by fractional Brownian motion \cite{kleptsyna2003lqregulator,duncan2013linear, duncan2017stochastic}, these often require  fractional integrals to obtain the optimal controls.

Over the last decade, machine learning methods have brought renewed optimism for high-dimensional control, largely by targeting the HJB equation (or related PDE objects) through function approximation \cite{sirignano2018dgm}. While these developments have produced impressive results in certain regimes, they still inherit the conceptual bottleneck of the HJB route: one approximates the  value function first, and only then recovers the optimal control (or its approximation). This motivates complementary approaches that parametrize controls directly in a rich, high-dimensional feature space and solve the resulting optimisation problem without passing through a PDE.

In this paper, we pursue such an alternative by using {signatures} of the driving noise to parametrize admissible controls. The signature of a path, introduced by Chen \cite{chen1954iterated} and developed systematically in rough path theory \cite{lyons2007differential}, is a collection of iterated integrals which provides a faithful and algebraically structured representation of the path \cite{hambly2010uniqueness, boedihardjo2016uniqueness}. In particular, linear functionals of (time-augmented) signatures form an algebra \cite{lyons2007differential} and enjoy strong approximation properties on suitable path spaces, making them a natural candidate for constructing expressive families of non-anticipative functionals. In stochastic analysis, iterated stochastic integrals underpin stochastic Taylor expansions \cite{platen1980approx,platen1982taylor,kloeden1992stochastic}, and signature theory consolidates these expansions into a unified algebraic framework in which truncation and polynomial structure can be handled systematically.

The idea of restricting to controls that are linear functionals of signatures (or expected signatures) has recently been explored in concrete stochastic optimisation tasks such as optimal execution \cite{arribas2020optimal,cartea2022double}. A key theoretical step was taken in \cite{bank2025stochastic}, which proves convergence of value functions when the objective functions are bounded in a way that permits dominated convergence arguments. While this yields a clean theory and accommodates rougher drivers, the boundedness requirement rules out the canonical LQ setting, despite LQ problems being a standard benchmark class. Nevertheless, the LQ structure suggests that a signature approach should be feasible if one can represent both the controlled state and the cost in a tractable signature-based form. Thus, our goal in this paper is to develop a rigorous signature-driven numerical scheme for multi-dimensional LQ stochastic control. 

Our starting point is that linear signature functionals are dense in the natural class of admissible controls for the LQ problem driven by Brownian noise, based on the recent density results in \cite{ceylan2025universality}. This allows us to restrict attention to signature controls without losing optimality in the limit. The second ingredient is an algebraic representation viewpoint, inspired by ideas sketched in \cite{arribas2020optimal}: by working with time-augmented signatures and their tensor algebra, we  represent the controlled state process (at least for Brownian drivers) via an infinite-dimensional linear functional acting on the signature of the driving noise, and we express the quadratic cost functional in the same algebraic language.  The last ingredient is the celebrated Fawcett formula which established a closed-form expression for the expected signature of Brownian motion \cite{fawcett2003problems, lyons2004cubature, ni2012expected}. As a result, once the control is parameterised as a linear functional of the (truncated) signature, the original stochastic control problem reduces to a deterministic convex quadratic polynomial optimisation in finitely many parameters.

To support a practical numerical method, we truncate (i) the tensor representation of the state at some level $L$, and (ii) the signature features used to parametrize controls up to level $M$. A central theoretical question is whether such finite-dimensional truncations yield value convergence as $L,M\to\infty$. Here we leverage recent growth and convergence estimates for tensor/signature expansions (in particular the growth estimates developed in \cite{jaber2024pathdependent}) to prove a rigorous consistency result: for  high truncation levels, the value function of the truncated polynomial optimisation approximates the true LQ value function. Importantly, this provides a convergence statement at the level of value functions, not merely almost-sure approximation of candidate controls.

Our numerical experiments confirm that the method we propose converges to the known LQ solution and, strikingly, achieves good accuracy for small truncation levels. Finally, the methodology itself is not tied to Brownian drivers: it extends naturally to other  signals such as fractional Brownian motion, where PDE methods are less directly applicable. Establishing a fully rigorous approximation theorem in such settings would require appropriate signature growth estimates for the chosen driver, which lies beyond the scope of the present paper. Nevertheless, our experiments indicate that the signature-driven optimisation remains effective in these non-Markovian regimes.
In summary, we make the following contributions:
\begin{enumerate}[wide, labelindent=5pt,label=(\roman*)]
  \item We develop a signature-driven numerical scheme for multi-dimensional LQ stochastic control by parametrizing controls as linear functionals of (time-augmented) path signatures.
  \item Using density of linear signature functionals in the admissible control class, we justify the signature-control  as an asymptotically non-restrictive approximation.
  \item We show that, under signature truncation, the LQ problem reduces to a deterministic convex quadratic polynomial optimisation problem.
  \item Using tensor/signature convergence and growth estimates, we prove consistency of the finite-dimensional truncations and obtain value-function approximation results as truncation levels increase.
  \item We provide numerical experiments demonstrating convergence to the known LQ solution with good accuracy at low truncation levels, and we discuss extensions beyond Brownian drivers (e.g.\ fractional Brownian motion).
\end{enumerate}

\paragraph{Related literature}
Signature-based approaches now appear in a growing range of stochastic optimization tasks. Examples include algorithmic trading  \cite{arribas2020optimal,cartea2022double}, the optimal stopping problem \cite{bayer2023optimal,bayer2025primal}  (the latter used dynamically normalised robust signatures \cite{chevyrev2022signature}), portfolio construction and trading
\cite{arribas2020sigsde,futter2023signature,futter2025kernel},  non-Markovian control problems \cite{hoglund2023neuralrde}, and the pricing and hedging of derivatives \cite{lyons2020non,jaber2024signature,jaber2025hedging,jaber2025signatureapproach}. On the theoretical side, robust signatures were introduced in \cite{chevyrev2022signature}, while global approximation and universality results for signature feature maps have been advanced in \cite{cuchiero2025global,ceylan2025universality}. The present paper fits into this broader program by establishing theoretical results about the applicability of signature methods to solve LQ stochastic control problems.

The remainder of the paper proceeds as follows. 
Section~\ref{sec:intro to signatures} reviews the signature background used throughout. Section~\ref{sec:problem formulation and results} introduces the control problem, the signature 
parametrization, and states our main theoretical results.  Section~\ref{sec:numerics} presents numerical experiments illustrating the performance of the proposed scheme.

\section{Some notations for dealing with path signatures} \label{sec:intro to signatures}
Let us briefly present some concepts and notation from Lyons's theory of path signatures; for further details, see for instance \cite{arribas2020optimal, cass2024lecturenotesroughpaths, chevyrev2026aprimer}.

\subsection{Tensor algebra notations}

Let $D\in \N$, $\Dbar = D+1$ and $\{e_1, \dots, e_{\Dbar}\}$ be the standard basis of $\R^{\Dbar}$ which induces the canonical dual basis $\{e_1^*, \dots, e_\Dbar^*\}$. Consider the algebra of tensors $$T(\R^{\Dbar}) := \left\{ (g_m)_{m=0}^\infty\, :\, g_m\in (\R^{\Dbar})^{\otimes m} \text{ and } \exists M\in \N \text{ such that } g_m = 0\, \forall m> M\right\},$$
the space of truncated tensors $$T^{ M}(\R^{\Dbar}) := \left\{ (g_m)_{m=0}^\infty\, :\, g_m\in (\R^{\Dbar})^{\otimes m} \text{ such that } g_m = 0\, \forall m> M\right\}\,,$$
and the space of extended tensors $$T^\ext(\R^{\Dbar}) := \left\{ (g_m)_{m=0}^\infty\, :\, g_m\in (\R^{\Dbar})^{\otimes m} \right\}\,,$$ where ``$\otimes$" denotes the tensor product.\footnote{ The space $(\R^{\Dbar})^{\otimes m}$ is isomorphic to $(\R^{(\Dbar)^m})$; for example, we can rearrange the matrix elements of $(\R^{\Dbar})^{\otimes 2}$ into a long vector in $(\R^{(\Dbar)^2})$.} We define $T((\R^{\Dbar})^*)$, $T^M((\R^{\Dbar})^*)$, and $T^\ext((\R^{\Dbar})^*)$ similarly by replacing $\R^\Dbar$ with its dual $(\R^\Dbar)^*$. Moreover, for any $g = (g_m)_{m=0}^\infty\in T^\ext(\R^\Dbar)$, we write its projection onto the truncated space $T^{ M}(\R^{\Dbar})$ as $g^{\leq M}$, that is
\begin{align}
    g^{\leq M} &= (\hat g_m)_{m=0}^\infty\,,\qquad
    \hat g_m =
    \begin{cases}
        g_m &\text{if } m\leq M\,,\\
        0 & \text{otherwise}\,.
    \end{cases}
\end{align}
For any $\ell \in T^\ext((\R^{\Dbar})^*)$, we define $\ell^{\leq M}$ similarly.

Throughout the paper, we represent the canonical dual basis as letters and words with blue characters. A letter represents an element from the canonical basis of $(\R^{\Dbar})^*$ and an $m$-word letter represents an element from the canonical basis of $((\R^{\Dbar})^*)^{\otimes m}$. Let us consider a set of $\Dbar$ letters $\mcW_\Dbar := \{\eone, \etwo, \dots, \eD, \eDbar\}$ with the convention that $\eDbar$ is the $\Dbar$-th letter. We associate $e_1^*$ with $\alphabet(1):=\eone$, $\alphabet(2):=e_2^*$ with $\etwo$, and so on until $e_\Dbar^*$ is represented by $\alphabet(\Dbar):=\eDbar$. With this, we identify the canonical basis of $((\R^{\Dbar})^*)^{\otimes m}$ with the set of $m$-letter words
\[V_m:=\{ \wb_1\,\dots\,\wb_m: (\wb_1, \dots, \wb_m) \in (\mcW_\Dbar)^m \}\,.\]
For any $\ell \in T^\ext((\R^{\Dbar})^*)$, $m\in \N$ and $\wb\in V_m$, we denote as $\ell^{\wb}$ the coefficient resulting from the projection of $\ell_m$ to the coordinate associated with the word $\wb$. We have that $$\ell = \sum_{m=0}^\infty \sum_{\wb\in V_m} \ell^\wb\, \wb\,.$$

We now define the Hilbert-Schmidt pairing $\langle \cdot, \cdot\rangle: T((\R^\Dbar)^*)\times T^\ext(\R^\Dbar)\to \R$ as
\begin{equation}
    \langle \ell, g\rangle := \sum_{m=0}^\infty \langle \ell_m, g_m \rangle_m, \quad \ell=(\ell_m)_{m=0}^\infty\in T((\R^\Dbar)^*),\quad \,g=(g_m)_{m=0}^\infty\in T((\R^\Dbar))\,,
\end{equation}
where $\langle \cdot, \cdot \rangle_m: ((\R^\Dbar)^*)^{\otimes m}\times (\R^\Dbar)^{\otimes m} \to \R$ is the natural dual pairing induced from the dot product in $(\R^\Dbar)^{\otimes m}$. The above pairing is a finite summation as $\ell$ has only finitely-many nonzero components. For $\ell\in T^\ext((\R^\Dbar)^*)$, we  extend the  pairing to
\begin{equation}
    \langle \ell, g\rangle := \lim_{M\to \infty}\sum_{m=0}^M \langle \ell_m, g_m \rangle_m
\end{equation}
whenever the expression converges.
Next, we  define the seminorm of $\ell\in T^\ext((\R^{\Dbar})^*)$ implied by the tensor $g=(g_m)_{m=1}^\infty\in T^\ext(\R^\Dbar)$ as
\begin{equation}
\label{eq: extended_tensor_seminorm}
    \|\ell\|_{T^\ext((\R^\Dbar)^*), g} := \sum_{m=0}^\infty \left| \langle \ell_m, g_m \rangle_m \right| = \sum_{m=0}^\infty \left| \sum_{\wb \in V_m} \ell^\wb\, g^\wb \right|\,,
\end{equation}
where $g^\wb := \langle \wb, g \rangle $.

\section{Problem formulation}\label{sec:problem formulation and results}

In this section, our aim is to express, algebraically, and as a (potentially infinite) linear combination of coefficients involving the driving noise signatures: (i) the controlled multidimensional state process and (ii) the linear-quadratic cost functions. Afterwards, we propose a truncation method to approximate the controlled system and the associated cost with a finite sum. We will show that optimising for these finite-dimensional approximations is a good substitute to the actual infinite-dimensional problem.

\subsection{Class of admissible tensor functionals}

Let $T>0$ and $(\Omega, \F, \mbF = (\mcF)_{t\in [0,T]}, \Pb)$ be a filtered probability space satisfying the usual conditions. Let $W = (W^{(1)}, \dots, W^{(D)})$ be a $D$-dimensional continuous $(\mbF, \Pb)$-progressively measurable process and $\Wsig_t$ be its time-extended signature up until time $t$, that is 
{\scriptsize
\begin{align}
 \Wsig_t &:= \left(\rule{0cm}{1.2cm}\right.
     1, 
     \begin{pmatrix}
         t\\ W_t^{(1)} \\ \vdots \\ W_t^{(D)}
     \end{pmatrix},
     \begin{pmatrix}
         \frac{t^2}{2!} & \int_0^t s\, \circ \d W_s^{(1)} & \cdots &  \int_0^t s\, \circ \d W_s^{(D)}  \\
         \int_0^t W_s^{(1)}\, \d s & \frac{\left(W_t^{(1)}\right)^2}{2!} & \cdots &  \int_0^t W_s^{(1)}\, \circ \d W_s^{(D)} \\
         \vdots & \vdots & \ddots & \vdots\\
         \int_0^t W_s^{(D)}\, \d s & \int_0^t W_s^{(D)}\, \circ \d W_s^{(1)} & \cdots &  \frac{\left(W_t^{(D)}\right)^2}{2!}
     \end{pmatrix},
     \underbrace{\begin{pmatrix}
         \cdots
     \end{pmatrix}}_{\substack{(D+1)^3\\ \text{ terms}}},
     \cdots\left.\rule{0cm}{1.2cm}\right)
\end{align}
}%
where ``$\int \circ$" denotes the Stratonovich integration. The signature $\Wsig_t$ lies on the space of extended tensors $T^\ext(\R^\Dbar)$ with the following coordinates: for any word $\wb = \alphabet(n_1) \cdots \alphabet(n_m) $ with $(n_1, \dots, n_m)\in \{1, \dots, \Dbar\}^m$, we have
\begin{equation}
    \Wsig_t^\wb = \int_0^t \int_0^{t_{m-1}} \cdots \int_0^{t_{1}} \d W_{t_1}^{(n_1 -1)} \circ \cdots \circ \d W_{t_m}^{(n_m -1)}\,,
\end{equation}
with the convention $\d W_t^{(0)} := \d t$. From rough path theory, the definition of signature above is well-defined for any continuous process, e.g., Brownian motion and fractional Brownian motion.
We then define the following subset of admissible tensors
\begin{align}
    \tensorLspace := \left\{ \ell \in T^\ext((\R^{\Dbar})^*):\, \Pb\left(\|\ell\|_{T^\ext((\R^{\Dbar})^*), \Wsig_t} < \infty \quad \forall t\in \mfT\right) = 1 \right\}\,.
\end{align}
It follows that $T((\R^{\Dbar})^*)\subset \tensorLspace$ as the sum in \eqref{eq: extended_tensor_seminorm} turns into a finite sum. Moreover, if $\ell\in \tensorLspace$, then the Hilbert-Schmidt pairing $\big(\langle \ell, \Wsig_t \rangle\big)_{t\in [0,T]}$,
\[ \langle \ell, \Wsig_t \rangle = \sum_{m=0}^\infty \sum_{\wb \in V_m} \ell^{\wb}\, \Wsig_t^\wb\,, \]
is a well-defined and progressively-measurable stochastic process.

It is well-known that the space of linear functionals acting on signatures forms an algebra under the so-called \textit{shuffle} product ``$\shpr$'' (refer to e.g. \cite{gaines1994thealgebra}; see \cite{jaber2024pathdependent} for the extended version). In particular, from \cite{jaber2024pathdependent} we have that if $\ell_1, \ell_2 \in \tensorLspace$, then $\ell_1 \shpr \ell_2\in \tensorLspace$ and
    \begin{align}
        \langle \ell_1, \Wsig_t \rangle\, \langle \ell_2\, \Wsig_t \rangle = \langle \ell_1 \shpr \ell_2, \Wsig_t\rangle\,.
    \end{align}

\subsection{Algebraic expression of a controlled multidimensional linear SDE}

We take the state $X$ to be $N$-dimensional and the control $\ct$ to be $K$-dimensional. Suppose that $\ct$ is a $\mbF$-progressively measurable process such that the system of stochastic differential equations (SDEs),
\begin{align}
    \d \st_t^{\ct, (n)} &= \left[ b_0^{(n)} + \sum_{k=1}^K b_1^{(n, k)}\, \ct_t^{(k)} + \sum_{n'=1}^N b_2^{(n, n')} \, \st_t^{\ct,(n')} \right]\,\d t\\
    &\qquad\qquad  + \sum_{d=1}^D \left[ \sigma_0^{(n, d)} + \sum_{n'=1}^N \sigma_2^{(n, d, n')}\, \st_t^{\ct,(n')} \right] \circ \d W_t^{(d)}\,, \qquad \st_0^{\ct, (n)} = \st_0^{(n)}\,,\label{eq:multidim_system}
\end{align}
admits a strong solution. For simplicity, all coefficients ($b_0^{(n, k)}$, $b_1^{(n, k)}$, $b_2^{(n, n')}$, $\sigma_0^{(n, d)}$, $\sigma_2^{(n, d, n')}$ for $n,n'=1,\dots, N$, $k=1,\dots, K$, $d=1, \dots D$) are constants.\footnote{This can be relaxed into deterministic polynomial-in-time coefficients}

Throughout this paper, we use parentheses in the superscripts to refer to a particular element of a vector/matrix/tensor. For example, $\st_t^{\ct, (n)}$ denotes the $n$-th element of the $N$-dimensional vector $\st_t^{\ct}$ and $\sigma_2^{(n,d,n')}$ denote the element of $\sigma_2$ in coordinate $(n, d, n')$. We include $\ct$ in the superscript of $X$ to draw attention to the dependence on the control.

We are particularly interested in controls $\ct$ which are given by linear functions of the signature:
\begin{align}
    \mcA_{\Sig} := \bigg\{ \ct: [0,T]\times \Omega \to \R^K \;\bigg|\; & \ct^{(k)}_t = \la \ctt^{(k)}, \Wsig_t \ra, \; k=1,\dots,K,\; t \in [0,T], \\&\text{ for some } \ctt = (\ctt^{(1)},  \dots, \ctt^{(K)}) \in (T((\R^{\Dbar})^*))^K \bigg\}\, .
\end{align}
As we will see, these signature controls form a sufficiently rich class for linear quadratic optimization problems with the clear advantage that, conveniently for numerically purposes, it is parametrized by elements $\ctt$ of the tensor algebra $(T((\R^{\Dbar})^*))^K$. 

The main result of the present section will be that for signature controls $\ct$ also their induced state dynamics $\st^\ct$ of~\eqref{eq:multidim_system} are of signature form, albeit with suitable elements $\stt^\ctt=(\stt^{\ctt,(1)},\dots,\stt^{\ctt,({N})})$ from the extended tensor algebra $(T^\ext((\R^{\Dbar})^*))^N$ which can be computed from $\ctt$. So, let us make the ansatz 
$$
\st_t^{\ct, (n)} = \langle \stt^{\ctt,(n)}, \Wsig_t  \rangle \text{ for some } \stt^{\ctt,(n)}\in T^\ext((\R^{\Dbar})^*), \quad n=1,\dots,N.
$$ 
Granted this ansatz works, we exploit the linearity of our system dynamics to express both sides of~\eqref{eq:multidim_system} using tensors
\begin{equation}
\label{eq:tensor_equation_1}
    \langle \stt^{\ctt,(n)}, \Wsig_t \rangle = \la \mathtt{p}^{\ctt,(n)} + \sum_{n'=1}^N \stt^{\ctt,(n')} \otimes \mathtt{q}^{(n,n')} , \Wsig_t \ra\,, \quad n=1,\dots,N,
\end{equation}
with
\begin{equation}
\label{eq:tensor_equation_2}
    \begin{split}
         \mathtt{p}^{\ctt,(n)} &:= \st_0^{(n)}\, {\emptyword} + \left(b_0^{(n)}{\emptyword}+\sum_{k=1}^K b_1^{(n,k)}\, \ctt^{(k)}\right)\otimes \eone + \sum_{d=1}^D \sigma_0^{(n,d)}\, \alphabet(d)\,,\\
         \mathtt{q}^{(n,n')} &:= b_2^{(n,n')}\, \eone + \sum_{d=1}^D \sigma_2^{(n,d,n')}\, \alphabet(d)\,.
    \end{split}
\end{equation}
We are thus led to consider the system of tensor equations
\begin{equation}
\label{eq:tensor_equation_3}
    \stt^{\ctt,(n)} =\mathtt{p}^{\ctt,(n)} + \sum_{n'=1}^N \stt^{\ctt,(n')} \otimes \mathtt{q}^{(n,n')}\,, \quad n=1,\dots,N.
\end{equation}
Our first result deals with the existence and uniqueness of the solution to this system.
\begin{proposition}
\label{prop:tensor_equation_sol}
    For any $\mathtt{p}^{(n)}, \mathtt{q}^{(n,n')} \in T^\ext((\R^\Dbar)^*)$, $n,n'=1,\dots,N$, with $(q^{(n,n')})^\emptyword=0$, there exists a unique $\stt=(\stt^{(1)},\cdots, \stt^{(N)}) \in T^\ext((\R^\Dbar)^*)^N$ such that
    \begin{equation}
    \label{eq: tensor_equation_shuffle}
        \stt^{(n)} =\mathtt{p}^{(n)} + \sum_{n'=1}^N \stt^{(n')} \otimes \mathtt{q}^{(n,n')}\,, \quad n=1,\dots,N.
    \end{equation}
\end{proposition}
\begin{proof}
    The problem is equivalent to solving the system of equations
    \begin{equation}
        \big(\stt^{(n)}\big)^\wb = \big(\mathtt{p}^{(n)}\big)^\wb + \sum_{n'=1}^N \big(\stt^{(n')} \otimes \mathtt{q}^{(n,n')}\big)^\wb
    \end{equation}
    parametrized by $n=1,\dots,N$ and arbitrary words $\wb$. For $\wb =\emptyword$, we have $(\stt^{(n')}\otimes \mathtt{q}^{(n,n')})^{\emptyword} = 0$ as $(\mathtt{q}^{(n,n')})^\emptyword=0$, implying that
    $$\big(\stt^{(n)}\big)^\emptyword = \big(\mathtt{p}^{(n)}\big)^\emptyword\,$$ is the only choice.
    Next, assume that we have solved for $\big(\stt^{(n)}\big)^\wb$ for all $n=1,\dots, N$ and any $|\wb|\leq M$. For $\jb\in \{\eone, \etwo, \cdots, \eDbar\}$, again we use the fact that $q^{(n,n')})^{\emptyword} = 0$ and obtain
    \begin{align}
        \big(\stt^{(n)}\big)^{\wb \jb} &= \big(\mathtt{p}^{(n)}\big)^{\wb \jb} + \sum_{n'=1}^N \sum_{\substack{\wb_1\wb_2 = \wb\jb\\ |\wb_1| \leq M}} \big(\stt^{(n')}\big)^{\wb_1}\, \big(\mathtt{q}^{(n,n')}\big)^{\wb_2}\,.
    \end{align}
    Thus, we would have solved for $\big(\stt^{(n)}\big)^{\vb}$ for all $n=1,\dots,N$ and word $\vb$ of length $(N+1)$. By induction, we find a solution $\stt = (\stt^{(1)},\dots,\stt^{(N)})$ that satisfies Equation \eqref{eq: tensor_equation_shuffle}. The uniqueness follows directly from the algorithm for computing each element of $\stt^{(n)}$.
\end{proof}

From Proposition \ref{prop:tensor_equation_sol}, the tensor equation~\eqref{eq:tensor_equation_3} has a unique solution $\stt^\ctt$ in $(T^\ext((\R^\Dbar)^*))^N$. However in general, even if the extended tensor $\stt^{\ctt, (n)}$ belongs in the space $\tensorLspace$ (which is not yet established at this point), it is not clear that the random process $\langle \stt^{\ctt, (n)}, \Wsig_\cdot \rangle$ replicates the original controlled state $\st^{\ct, (n)}$. Moreover, it remains unknown that the pairing is integrable (either with respect to $\d t$, $\d \Pb$, or both). Even if it is integrable, it does not follow automatically that one can move the expectation inside the pairing, i.e., $\E [\langle \stt^{\ctt, (n)}, \Wsig_t\rangle ] = \langle \stt^{\ctt, (n)}, \E [\Wsig_t] \rangle$. 

All of the above complications are resolved when the driving noise $W$ is Brownian. The foundational work by \cite{jaber2024pathdependent} provides a growth condition for the entries of $\ell$ such that random processes induced by the tensor and the Brownian signatures are integrable in some sense.
\begin{lemma}
\label{lemma: exponential_growth_level_tensor}
    Define the set
    \begin{align}
        \tensorExpspace:=\left\{ \ell \in T^\ext((\R^\Dbar)^*)\,: \, \exists\, C\in (0,\infty) \text{ such that } \sup_{m\in \N} \sup_{\vb\in V_m} \frac{|\ell^{\vb}|}{C^{m}} <\infty\right\}\,.
    \end{align}
    The following statements hold. 
    \begin{enumerate}[wide, labelwidth=!, labelindent=10pt]
        \item \label{lemma: exponential_growth_level_tensor_1} $\tensorExpspace$ is a linear space that is closed under tensor and shuffle product, i.e. for any $\ell_1, \ell_2\in \tensorExpspace$, $a, b\in \R$, we have $a\,\ell_1 + b\, \ell_2\,, \ell_1\otimes \ell_2\,, \ell_1\shpr \ell_2 \in \tensorExpspace$.
        \item If the driving noise $W$ is Brownian, for any $\ell, \ell_1, \ell_2 \in \tensorExpspace$, we have
        \begin{enumerate}
            \item \label{lemma: exponential_growth_level_tensor_2} $\tensorExpspace \subset \tensorLspace$  and $\displaystyle \E \bigg[ \sup_{t\in [0,T]} \big| \langle \ell, \Wsig_t \rangle|^2\bigg]  < \infty$,
            \item \label{lemma: tensor functional convergence} $\displaystyle \lim_{\lev\to \infty} \E \bigg[ \sup_{t\in [0,T]} \big| \langle \ell, \Wsig_t \rangle - \langle \ell^{\leq \lev}, \Wsig_t\rangle \big|^2\bigg] = 0\,,$
            \item  \label{lemma: tensor exchange operation} $\displaystyle \E \left[ \langle \ell, \Wsig_T\rangle \right] = \langle \ell, \E [ \Wsig_T] \rangle\,,$
            \item \label{lemma: tensor integral exchange operation} $\displaystyle \E \left[ \int_0^T \langle \ell, \Wsig_t\rangle\, \d t \right] = \langle \ell\otimes \eone, \E [ \Wsig_T] \rangle\,,$
            \item \label{lemma: shpr approx finite} $\displaystyle \lim_{\lev_1, \lev_2 \to \infty} \E \left[ \sup_{t\in [0,T]} \big| \langle \ell_1 \shpr \ell_2 - \ell_1^{\leq \lev_1} \shpr \ell_2^{\leq \lev_2}, \Wsig_t \rangle \big| \right] = 0,$
            \item \label{lemma: mean shpr approx finite} $\displaystyle \E [ \langle \ell_1, \Wsig_t \rangle\, \langle \ell_2, \Wsig_t \rangle] = \lim_{\lev_1, \lev_2 \to \infty} \langle \ell_1^{\leq \lev_1} \shpr \ell_2^{\leq \lev_2}, \E [\Wsig_t]\rangle\,,$ and
            \item \label{lemma: integral shpr approx finite} for all $m\in \N$,
            \begin{align}
                &\E \left[ \int_0^T \frac{t^m}{m!}\, \langle \ell_1, \Wsig_t \rangle\, \langle \ell_2, \Wsig_t \rangle\, \d t\right]\\
                &\qquad = \lim_{\lev_1, \lev_2\to \infty} \la (\eone^{\otimes m} \shpr \ell_1^{\leq \lev_1} \shpr \ell_2^{\leq \lev_2})\otimes \eone, \E [\Wsig_T]\ra\,.
            \end{align}
        \end{enumerate}
    \end{enumerate}
\end{lemma}
\begin{proof}
    First, we will prove that the space $\tensorExpspace$ is the same as the space of \textit{exponentially-dominated} tensors $\mcA_{\exp}$ introduced in \cite{jaber2024pathdependent}. Let $\ell\in \tensorExpspace$ with the corresponding constant $C\in (0,\infty)$. Define $\hat C := \max \left\{ C, \sup_{m\in \N} \sup_{\vb\in V_m} \frac{|\ell^{\vb}|}{C^{m}} \right\}$ and $\hat{\mathtt{C}}:=\hat C e^{\shpr \hat C \sum_{\jb \in \mcW_\Dbar} \jb} $. From Proposition 2.3 in  \cite{jaber2024pathdependent} we have that
    \begin{align}
        \hat{\mathtt{C}} = \hat C\, \sum_{m=0}^\infty \left(\hat C \sum_{\jb \in \mcW_\Dbar} \jb\right)^{\otimes m} = \sum_{m=0}^\infty \hat C^{m+1} \sum_{\vb \in V_m} \vb\,.
    \end{align}
    Thus, for any $m\in \N$ and $\vb\in V_m$, $|\ell^{\vb}| \leq \hat C\, C^m \leq \hat C^{m+1} = \hat{\mathtt{C}}^\vb\,.$
    Conversely, if there exists $\tilde C \in (0,\infty)$ such that $|\ell^{\vb}| \leq \left( \tilde C e^{\shpr \tilde C \sum_{\jb \in \mcW_\Dbar} \jb} \right)^\vb$ for any word $\vb$, then $\sup_{m\in \N} \sup_{\vb\in V_m} \frac{|\ell^{\vb}|}{\hat C^{m}} \leq \hat C < \infty$. We conclude that statements \ref{lemma: exponential_growth_level_tensor_1}, \ref{lemma: exponential_growth_level_tensor_2}, and \ref{lemma: tensor functional convergence} directly follow from Proposition 3.12 in \cite{jaber2024pathdependent}.
    For \ref{lemma: tensor exchange operation}, we have
    \begin{align}
        \Big|\E \big[ \langle \ell, \Wsig_T\rangle \big] - \langle \ell, \E [ \Wsig_T] \rangle \Big| &=\lim_{m \to \infty} \Big|\E \big[ \langle \ell, \Wsig_T\rangle \big] - \langle \ell^{\leq M}, \E [ \Wsig_T] \rangle \Big|\\
        &= \lim_{m \to \infty} \Big|\E \big[ \langle \ell, \Wsig_T\rangle - \langle \ell^{\leq M}, \Wsig_T\rangle \big] \Big|\\
        &\leq \lim_{m \to \infty} \E \big[ \big|\langle \ell, \Wsig_T\rangle - \langle \ell^{\leq M}, \Wsig_T\rangle \big| \big] = 0\,.
    \end{align}
    Statement \ref{lemma: tensor integral exchange operation} follows similarly. Next, observe that
    \begin{align}
       \ell_1 \shpr \ell_2 - \ell_1^{\leq \lev_1} \shpr \ell_2^{\leq \lev_2} &= \big(\ell_1 - \ell_1^{\leq \lev_1}\big) \shpr\ell_2 + \ell_1^{\leq \lev_1}\shpr \big( \ell_2  -\ell_2^{\leq \lev_2} \big)\,.
    \end{align}
    Thus, it follows that 
    \begin{align}
        \E \bigg[ &\sup_{t\in [0,T]} \big| \langle \ell_1 \shpr \ell_2 - \ell_1^{\leq \lev_1} \shpr \ell_2^{\leq \lev_2}, \Wsig_t \rangle \big| \bigg]\\
        &\leq
        \E \bigg[ \sup_{t\in [0,T]} \big|\langle (\ell_1 - \ell_1^{\leq \lev_1}) \shpr \ell_2, \Wsig_t \rangle \big| \bigg] +\E \bigg[ \sup_{t\in [0,T]} \big| \langle \ell_1^{\leq \lev_1}\shpr \big( \ell_2  -\ell_2^{\leq \lev_2} \big), \Wsig_t \rangle\big| \bigg]\,.\label{eq: inequality approx finite tensors}
    \end{align}
    By applying the extended shuffle property (Proposition 3.1 in \cite{jaber2024pathdependent}) and Cauchy-Schwarz inequality we have
        \begin{align}
            &\E \bigg[ \sup_{t\in [0,T]} \big|\langle (\ell_1 - \ell_1^{\leq \lev_1}) \shpr (\ell_2 - \ell_2^{\leq \lev_2}), \Wsig_t \rangle \big|\bigg]\\
              &\qquad  =\E \bigg[ \sup_{t\in [0,T]} \big|\langle \ell_1 - \ell_1^{\leq \lev_1}, \Wsig_t \rangle \langle \ell_2, \Wsig_t \rangle \big|\bigg]\\
            &\qquad  \leq \E \bigg[ \sup_{t\in [0,T]} \langle \ell_1 - \ell_1^{\leq \lev_1}, \Wsig_t \rangle^2\bigg]^{1/2}\, \E \bigg[\sup_{t\in [0,T]} \langle \ell_2, \Wsig_t \rangle^2\bigg]^{1/2} \xrightarrow{\lev_1 \to \infty}  0\,.
        \end{align}
    Moreover, observe that $\sup_{\lev_1 \in \N} \E \bigg[\sup_{t\in [0,T]} \langle \ell_1^{\leq \lev_1}, \Wsig_t \rangle^2\bigg] < \infty$. As such, by applying the extended shuffle product and Cauchy-Schwarz again, the second term in the right-hand side of \eqref{eq: inequality approx finite tensors} also converges to zero, proving \ref{lemma: shpr approx finite}.  Statement \ref{lemma: mean shpr approx finite} follows from the extended shuffle product and \ref{lemma: shpr approx finite}. Statement \ref{lemma: integral shpr approx finite} follows from the fact that $t^m/m! = \langle \eone^{\otimes m}, \Wsig_t\rangle$ and from applying the extended shuffle product, statement \ref{lemma: tensor integral exchange operation}, and statement \ref{lemma: shpr approx finite}.
\end{proof}
The following lemma shows conditions under which the solution $\stt^\ctt$ of~\eqref{eq:tensor_equation_3} belongs to the space $\tensorExpspace$.
\begin{lemma}
\label{lemma: convergence_tensor_multidim_linear}
    For $\ctt\in \big(T((\R^{\Dbar})^*)\big)^K$, the solution $\stt^{\ctt,(n)}$, $n=1,\dots, N$, of~\eqref{eq:tensor_equation_3} is an element of the space $\tensorExpspace$.
\end{lemma}
\begin{proof}
    Let us define $\hat{\mathtt{p}}^{\ctt, (n)}:= (\sum_{k=1}^K b_1^{(n,k)}\, \ctt^{(k)} ) \otimes \eone $. Then, for all $\vb\in V$ and $\jb\in \{\eone, \dots, \eD, \eDbar\}$, we have
    \begin{align}
        \big(\stt^{\ctt,(n)}\big)^{\vb \jb} &= \big(\mathtt{p}^{\ctt, (n)}\big)^{\vb \jb} + \sum_{n'=1}^N \big(\stt_\ctt^{(n')}\big)^{\vb}\, \big(\mathtt{q}^{(n,n')}\big)^{\jb}\\
        &=  \big(\mathtt{p}^{0, (n)}\big)^{\vb \jb} + \hat{\mathtt{p}}^{\ctt, (n)} + \sum_{n'=1}^N \big(\stt_\ctt^{(n')}\big)^{\vb}\, \big(\mathtt{q}^{(n,n')}\big)^{\jb}\,.
    \end{align}
    Recall that $\mathtt{p}^{0, (n)}$ and $q^{(n,n')}$ are linear combinations of 1-letter words. Moreover, as $\ctt$ lives in $\big(T((\R^{\Dbar})^*)\big)^K$, so does $\hat{\mathtt{p}}^{\ctt, (n)} $. Thus, at some point, $\big(\hat{\mathtt{p}}^{\ctt, (n)}\big)^\vb = 0$ for long words, i.e. there exists $M\in \N$ such that $\big(\hat{\mathtt{p}}^{\ctt, (n)}\big)^\vb = 0$ for all $n=1,\dots, N$, $\vb\in V$ with $|\vb|\geq M$. We define
    \[ C_1 := \exp \left\{ \max_{n=1,\dots N} \max_{m=1,\dots, M} \max_{\vb\in V_m} \frac{\log \left|\left(\stt^{\ctt,(n)}\right)^\vb\right|}{m}\right\}\,,\]
    and we take 
    $$ C := \max\left\{ C_1, \max_{\jb \in \mcW_D} \max_{n=1,\dots, N} \sum_{n'=1}^N \left|\left(\mathtt{q}^{(n,n')}\right)^\jb\right|\right\}\,. $$
    By this construction, we have
    \begin{equation}
        \left|\left(\stt^{\ctt,(n)}\right)^\vb\right| \leq C_1^{|\vb|} \leq C^{|\vb|}
    \end{equation}
    for all words $\vb$ with length $|\vb|\leq M$. Assume $\left|\left(\stt^{\ctt,(n)}\right)^\vb\right| \leq C^{|\vb|}$ for all words $\vb$ with length $|\vb| \leq \bar M$, $\bar M \geq M$. Then for any letter $\jb\in \mcW_D$,
    \begin{equation}
        \left|\left(\stt^{\ctt,(n)}\right)^{\vb\, \jb}\right| \leq C^{|\vb|} \sum_{n'=1}^N \left|\left(\mathtt{q}^{(n,n')}\right)^\jb\right| \leq C^{|\vb\, \jb|}\,.
    \end{equation}
    The lemma follows by induction.
\end{proof}
From lemmas \ref{lemma: exponential_growth_level_tensor} and \ref{lemma: convergence_tensor_multidim_linear}  we  obtain the following corollary.
\begin{corollary}
\label{cor: state process as linear functional of signature}
    Consider $\ctt\in \big(T((\R^{\Dbar})^*)\big)^K$ and let the corresponding signature control $\ct \in \mcA_{Sig}$ be given by $$\ct_t = \la \ctt, \Wsig_t \ra :=\big(\langle \ctt^{(1)}, \Wsig_t\rangle, \cdots, \langle \ctt^{(K)}, \Wsig_t\rangle \big)\,.$$ Then the controlled state process $\st^\ct$ from SDE~\eqref{eq:multidim_system} can be written as
    $$
    \st^\ct_t= \la \stt^\ctt, \Wsig_t \ra :=\big( \langle \stt^{\ctt,(1)}, \Wsig_t\rangle, \dots,  \langle \stt^{\ctt,(N)}, \Wsig_t\rangle\big), \quad {t\in[0,T]}\,,
    $$
    where $\stt^\ctt=(\stt^{\ctt,(1)},\dots,\stt^{\ctt,(N)}) \in (\tensorExpspace)^N$ is the unique collection of extended tensors solving~\eqref{eq:tensor_equation_3}.
\end{corollary}

In light of Lemma \ref{lemma: convergence_tensor_multidim_linear}, we will assume the driving noise to be Brownian from this point forward. Later, we will discuss  similar results for other processes (e.g. fractional Brownian motion) as well as the associated challenges involved.

\begin{assumption}
    The process $W$ is a $D$-dimensional standard Brownian motion.
\end{assumption}

\subsection{Linear-quadratic cost functional and its algebraic properties}

Next, we are going to express the linear-quadratic cost function as a pairing with the driving noise signature. Additionally, we are also going to describe the algebraic and convexity properties of the cost function.

\subsubsection{Tensor representation of the cost functional}

For linear state dynamics $\st^\ct$ as in~\eqref{eq:multidim_system}, a widely used cost functional in stochastic control problems is the linear-quadratic criterion
\begin{align}
    J(\ct) &:= \E \Bigg[ \int_0^T \Big\{ \big(\st_t^\ct\big)^\top\, A(t)\,  \st_t^\ct\, +  \ct_t^\top\, B(t)\, \ct_t \\
    \label{eq: original cost functional} &\hspace{7em} + 2\, C(t)^\top\, \st_t^\ct + 2\, D(t)^\top\, \ct_t\, \Big\} \d t + \big(\st_T^\ct\big)^\top\, E\,  \st_T^\ct + 2\, G^\top\, \st_T^\ct \Bigg],
\end{align}
for continuous matrix valued mappings $A,B,C,D$ on $[0,T]$ and matrices $E,G$ such that $A(t), E \in \R^{N \times N}$ are symmetric, positive semi-definite, $B(t) \in \R^{K \times K}$ is symmetric positive definite, and $C(t) \in \R^{N \times 1}$, $D \in \R^{K \times 1}$ for any $t \in [0,T]$. 

As we seek to minimize these costs, it is natural to focus on admissible controls $\ct$ from
\begin{align}
    \mcA &:= \bigg\{ \ct: [0,T]\times \Omega \to \R^K \,\bigg|\, \ct \text{ is $\mbF$-progressively measurable}\\
    & \hspace{15em} \text{and } \E \left[ \int_0^T |\ct_s|^2\, \d s\right] < \infty\bigg\}\,.
\end{align}
It follows that $\mcA_{\Sig}\subset \mcA$ as each entry of $\Wsig$ is square-integrable and the tensors $\ctt \in (T(\R^{D+1}))^K$ defining signature controls have only finitely many non-zero entries.

By Stone-Weierstra{\ss}, we can uniformly approximate the continuous matrix functions $A,B,C,D$ with polynomials in $t$. Thus, we are justified in considering the polynomial case when
\begin{align}
    A(t) := \sum_{m=0}^M \frac{t^m}{m!} A_{m}\,, \quad B(t) := \sum_{m=0}^M \frac{t^m}{m!} B_{m}\,,\\
    C(t):= \sum_{m=0}^M \frac{t^m}{m!} C_{m}\,, \quad D(t):= \sum_{m=0}^M \frac{t^m}{m!} D_{m}\\
\end{align}
where $A_m\in \R^{N\times N}$ and $B_m\in \R^{K\times K}$ are symmetric such that $A(t)\geq 0$, $B(t) > 0$ for all $t\in[0,T]$. Under this assumption, it turns out that not only is the state process from a signature control $\ct_t = \langle \ctt,\Wsig_t \rangle$ of the signature form $\st^\ct_t=\langle \stt^\ctt,\Wsig_t\rangle$, $t \in [0,T]$, but also the induced cost can be described by a linear functional on the extended tensor algebra, albeit applied to the {expected} signature $\E[\Wsig_T]$.  

\begin{proposition}
\label{prop: cost functional as linear functional}
    The costs incurred by a signature policy $\ct=\la \ctt,\Wsig\ra\in \mcA_\Sig$  can be written in the form
    \begin{align}
    \label{eq:perf_criterion_sig}
        J(\ct) =\la \mfq(\ctt), \E\big[ \Wsig_T\big] \ra = \lim_{\lev\to \infty} \la \mfq^\lev(\ctt), \E \big[ \Wsig_T\big] \ra\,,
    \end{align}
    where, for $\ctt \in \big(T((\R^{\Dbar})^*)\big)^K$, $\mfq(\ctt)$ is in $\tensorExpspace$ with the following explicit expression,
    \begin{align}
        \label{eq: cost functional tensor}
        \mfq(\ctt) := \Bigg\{ &\sum_{m=0}^M \sum_{ n, n'=1}^N A_m^{(n,n')}\, \eone^{\otimes m} \shpr \stt^{\ctt,(n)} \shpr \stt^{\ctt,(n')}\\
        &\quad + \sum_{m=0}^M \sum_{k, k'=1}^K B_m^{(k,k')}\,\eone^{\otimes m} \shpr \ctt^{(k)} \shpr \ctt^{(k')}\\
        &\quad+ 2\, \sum_{m=0}^M \sum_{n=1}^N C_m^{(n)}\, \eone^{\otimes m} \shpr \stt^{\ctt,(n)}\\
        &\quad + 2\, \sum_{m=0}^M \sum_{k=1}^K D_m^{(k)}\, \eone^{\otimes m} \shpr \ctt^{(k)} \Bigg\} \otimes \eone\\
        & + \sum_{n, n'=1}^N E^{(n,n')}\, \stt^{\ctt,(n)} \shpr \stt^{\ctt,(n')} + 2\,  \sum_{n=1}^N G^{(n)}\, \stt^{\ctt,(n)}\,,
    \end{align}
    and $\mfq^\lev(\ctt)\in T((\R^\Dbar)^*)$ is a version of $\mfq(\ctt)$ obtained by replacing $\stt^{\ctt, (n)}$ with its truncated version $\stt^{\ctt, (n), \lev} := \big(\stt^{\ctt, (n)}\big)^{\leq \lev}$.
\end{proposition}

The algebraic tensor representation in \eqref{eq:perf_criterion_sig} is quite  powerful: we have transformed the stochastic optimal control problem into a deterministic one. In practice, we cannot numerically compute $\mfq(\ctt)$ fully as it involves infinite dimensional tensors $\stt^{\ctt,(n)}$. A sensible approximation in this situation is to truncate $\stt^{\ctt,(n)}$ first at some level $\lev$, then use this as a proxy for the full tensor $\stt^{\ctt, (n)}$ and, finally, consider the approximate criterion $\E \left[\langle \mfq^\lev(\ctt), \Wsig_T \rangle\right]$
where $\mfq^\lev(\ctt)$ is determined by replacing $\stt^{\ctt, (n)}$ with the truncated version $\stt^{\ctt, (n), \lev}$. Observe that $\mfq^\lev (\ctt)$ now lives in the space $T((\R^\Dbar)^*)$. As such, $\langle \mfq^\lev(\ctt),  \Wsig_T\rangle$ is a linear combination of finitely-many elements of $\Wsig_T$ and thus $\E \left[\langle \mfq^\lev(\ctt), \Wsig_T \rangle\right] = \langle \mfq^\lev(\ctt), \E[\Wsig_T] \rangle$.

\begin{proof}[Proof of Proposition \ref{prop: cost functional as linear functional}]
    From the extended shuffle product and the statements of Lemma \ref{lemma: exponential_growth_level_tensor}, we are able to write the cost functional as follows:
    \begin{align}
        J(\ct) &= \E \Bigg[ \int_0^T \Bigg\{ \quad \sum_{m=0}^M \sum_{ n, n'=1}^N A_m^{(n,n')}\, \frac{t^m}{m!}\, \st^{\ct, (n)}_t\, \st^{\ct, (n')}_t\\
        &\qquad \qquad \qquad + \sum_{m=0}^M \sum_{k, k'=1}^K B_m^{(k,k')}\,\frac{t^m}{m!}\, \ct^{(k)}_t\, \ct^{(k')}_t\\
        &\qquad \qquad \qquad + 2\, \sum_{m=0}^M \sum_{n=1}^N C_m^{(n)}\, \frac{t^m}{m!}\, \st^{\ct, (n)}_t\\
        &\qquad \qquad \qquad + 2\, \sum_{m=0}^M \sum_{k=1}^K D_m^{(k)}\, \frac{t^m}{m!} \, \ct^{(k)}_t \Bigg\} \, \d t\\
        &\qquad \quad +  \sum_{n, n'=1}^N E^{(n,n')}\, \st^{\ct, (n)}_T\, \st^{\ct, (n')}_T + 2\,  \sum_{n=1}^N G^{(n)}\, \st^{\ct, (n)}_T\Bigg]\\
        &= \langle \mfq(\ctt), \E[\Wsig_T]\rangle\,,
    \end{align}
    where $\mfq(\ctt)$ is given by \eqref{eq: cost functional tensor}. As $\mfq(\ctt)$ is a linear combination of finitely-many tensors in $\tensorExpspace$, the tensor also belong in this space. From Lemma \ref{lemma: exponential_growth_level_tensor} (statement \ref{lemma: mean shpr approx finite} and \ref{lemma: integral shpr approx finite}), we have
    \begin{align}
        \langle \mfq(\ctt), \E[\Wsig_T]\rangle = \lim_{\lev \to \infty} \langle \mfq^\lev(\ctt), \E[\Wsig_T]\rangle\,.
    \end{align}
\end{proof}
Next, we seek to investigate the algebraic properties of the cost function $J(\langle \cdot, \Wsig\rangle) = \langle \mfq(\cdot), \E[\Wsig_T]\rangle$.
\begin{lemma}
\label{lemma: polynomial lemma 1}
    Let $L_1, L_2: \big(T^{ M} ((\R^\Dbar)^*)\big)^K \to \big(T^\ext((\R^\Dbar)^*)\big)^N$ be mappings such that for $i\in\{1,2\}$, $L_i$ can be written as
    \begin{align}\label{eq:Lrecursion}
    L_i(\cdot) = P_{i,0} + P_{i,1}(\cdot) + P_{i,2}(L_i(\cdot))
    \end{align}
    for some $P_{i,0}\in \big(T((\R^\Dbar)^*)\big)^N$, linear maps $P_{i,1}: \big(T^{ M} ((\R^\Dbar)^*)\big)^K \to \big(T((\R^\Dbar)^*)\big)^N$ and $P_{i,2}: \big(T^\ext((\R^\Dbar)^*)\big)^N \to \big(T^\ext((\R^\Dbar)^*)\big)^N$
    with 
    \begin{align}
        \Proj_j^N \circ \underbrace{P_{i,2} \circ \dots \circ P_{i,2}}_{\text{$j$ times}} = 0\quad \forall j\in \N.  \label{eq:shift operator}
    \end{align} Here $\Proj_j^N$ is the projection map from $\big(T^\ext((\R^\Dbar)^*)\big)^N$ to $\big(T^{j}((\R^\Dbar)^*)\big)^N$. It follows that:
    \begin{enumerate}[wide, labelwidth=!, labelindent=10pt]
        \item \label{lemma: statement 1} For every $i\in \{1,2\}$, $n\in 1,\dots, N$, $\lev \in \N$, and $G\in T^{\ext}(\R^\Dbar)$, the expression
        $$\la \big(L_i(\ctt)^{(n)}\big)^{\leq \lev}, G \ra$$
        is an affine function in terms of the coefficients of $\ctt$.
        \item \label{lemma: statement 2}  For every $n,n'=1,\dots, N$, $G\in T^\ext(\R^\Dbar)$, $\lev_1, \lev_2 \in \N$, and linear map $P: T((\R^\Dbar)^*)\to T((\R^\Dbar)^*)$, the expression
        $$\la P \Big( \big(L_1(\ctt)^{(n)}\big)^{\leq \lev_1}\shpr \big(L_2(\ctt)^{(n')}\big)^{\lev_2} \Big), G \ra$$
        is a quadratic polynomial in terms of the coefficients of $\ctt$.
    \end{enumerate}
\end{lemma}
\begin{proof} 
To prove statement \ref{lemma: statement 1}, we have that by repeatedly iterating equation~\eqref{eq:Lrecursion}
        \begin{align}
            L(\ctt) &= \sum_{m=0}^{\lev-1} (\underbrace{P_2\circ \dots \circ P_2}_{\text{$m$ times}})(P_0) + (\underbrace{P_2\circ \dots \circ P_2}_{\text{$m$ times}}\circ\, P_1)(\ctt)\\
            &\qquad\qquad  + (\underbrace{P_2\circ \dots \circ P_2}_{\text{$\lev$ times}})(L(\ctt)))\,.
        \end{align}
        However, we have $\Proj_{\lev}^N \circ \underbrace{P_2 \circ \dots \circ P_2}_{\text{$\lev$ times}} = 0$. Thus, 
        \[ \Proj_{\lev}^N (L(\ctt)) = \Proj_{\lev}^N \left(\sum_{m=0}^{\lev-1} (\underbrace{P_2\circ \dots \circ P_2}_{\text{$m$ times}})(P_0) + (\underbrace{P_2\circ \dots \circ P_2}_{\text{$m$ times}}\circ\, P_1)(\ctt)\right)\,. \]
        As  $\Proj_{\lev}^N \circ \overbrace{P_2\circ \dots \circ P_2}^{\text{$\lev$ times}}\circ\, P_1$ is a linear map acting on $\ctt$, it follows that  $\langle \big(L(\ctt)^{(n)}\big)^{\leq \lev}, G \rangle$ is an affine function in terms of the coefficients of $\ctt$.
As for statement \ref{lemma: statement 2},  observe that $\la P \Big( \big(L_1(\ctt)^{(n)}\big)^{\leq \lev_1}\shpr \big(L_2(\ctt)^{(n')}\big)^{\lev_2} \Big), G \ra$ is a linear functional over the coefficients of $\big(L_1(\ctt)^{(n)}\big)^{\leq \lev_1}\shpr \big(L_2(\ctt)^{(n')}\big)^{\lev_2}$. However, those coefficients have the form
        \begin{align}
            &\la \big(L_1(\ctt)^{(n)}\big)^{\leq \lev_1}\shpr \big(L_2(\ctt)^{(n')}\big)^{\lev_2}, \vb\ra\\
            &\qquad = \sum_{|\vb_1| + |\vb_2| = |\vb|} \la \big(L_1(\ctt)^{(n)}\big)^{\leq \lev_1}, \vb_1 \ra\, \la \big(L_2(\ctt)^{(n')}\big)^{\leq \lev_2}, \vb_2 \ra\, \langle \vb_1\shpr \vb_2, \vb \rangle\,
        \end{align}
        for any word $\vb$. From the previous statement, it follows that  $\la \big(L_1(\ctt)^{(n)}\big)^{\leq \lev_1}, \vb_1 \ra$ and $\la \big(L_2(\ctt)^{(n')}\big)^{\leq \lev_2}, \vb_2 \ra$ are affine functions of the coefficients of $\ctt$.
\end{proof}
In light of our tensor equation~\eqref{eq:tensor_equation_3} encoding the system dynamics, it is straightforward to see that the mapping $\ctt \mapsto \stt^\ctt$ satisfies the conditions of Lemma \ref{lemma: polynomial lemma 1}. As we can decompose $\langle \mfq^\lev(\ctt), \E [\Wsig_T]\rangle$ into a finite sum of terms in the form of $\langle P(\stt^{\ctt, (n), \lev} \shpr \stt^{\ctt, (n'), \lev}), G\rangle$ or $\langle \stt^{\ctt, (n), \lev}, G\rangle$, the expression $\la \mfq^\lev(\ctt), \E \big[ \Wsig_T\big] \ra$ is a quadratic polynomial. Additionally, Proposition \ref{prop: cost functional as linear functional} tells us that $$\la \mfq^\lev(\cdot), \E \big[ \Wsig_T\big] \ra \xrightarrow{\text{pointwise}} \la \mfq(\cdot),  \E[\Wsig_T] \ra\,,$$ so the actual cost functional is also a quadratic polynomial of the coefficients of $\ctt$.
\begin{corollary}
\label{lemma: full cost functional as polynomial}
    Let $\mfq(\cdot)$ be as in \eqref{eq: cost functional tensor}. Then, restricted to $\big(T^M((\R^{\Dbar})^*)\big)^K$, the functions $\langle \mfq^\lev(\cdot), \E\big[ \Wsig_T\big] \rangle$ and $\langle \mfq(\cdot), \E[\Wsig_T]\rangle$ are quadratic polynomials in terms of the tensor coefficients.
\end{corollary}

\subsubsection{Strict convexity of the cost functional}

The linear-quadratic cost function $J(\cdot)$ in \eqref{eq: original cost functional} is strictly convex in the space of controls $\mcA$; automatically, it is also convex in the subspace $\mcA_\Sig$. As we are working with the tensor representation of the control, we will also need the cost function $\langle \mfq(\cdot), \E[\Wsig_T] \rangle$ to be strictly convex in the space $\big(T((\R^\Dbar)^*)\big)^K$. To show this, we need the following lemma.
\begin{lemma}
\label{lemma: injectiveness of signature control}
    The empty word $\ell = \emptyword$ is the only $\ell\in T((\R^\Dbar)^*)$ such that $\langle \ell, \Wsig_t \rangle = 0$ for all $t\in [0,T]$,  $\Pb$-almost surely.
\end{lemma}
\begin{proof}
    Let us define
    \begin{align}
        \Nn = \left\{ \ell \in T((\R^\Dbar)^*)\, :\, \langle \ell, \Wsig_t \rangle = 0\quad \forall t\in [0,T]\quad \Pb\text{-almost surely}\right\}\,.
    \end{align}
    Consequently, $\Nn$ forms a sub-vector space of $T((\R^\Dbar)^*)$. Consider some $\ell\in \Nn$. For a word $\vb$, let us denote
    \begin{equation}
        \ell\, |_{\vb} := \sum_{m=0}^\infty \sum_{\wb \in V_m} \ell^{\wb\, \vb}\, \wb\,.
    \end{equation}
    Moreover, let us denote $W^{(\alphabet(j+1))} = W^{(j)}$ for $j=1, \dots, D$. By Ito's formula for pairings with Brownian signatures (see Theorem 3.3 in \cite{jaber2024pathdependent} for the extended pairing)
    \begin{align}
        0&= \langle \ell, \Wsig_t \rangle\\
        &= \ell^{\emptyword} + \int_0^t \langle \ell\, |_{\eone}, \Wsig_u \rangle\, \d u + \sum_{\jb\in \{\etwo, \cdots, \eDbar\}} \int_0^t \langle \ell\, |_{\jb}, \Wsig_u \rangle\, \circ \d W_u^{(\jb)}\\
        &= \ell^{\emptyword} + \int_0^t \la \ell\, |_{\eone}  + \frac{1}{2} \sum_{\jb\in \{\etwo, \cdots, \eDbar\}} \ell\, |_{\jb\, \jb}, \Wsig_u \ra\, \d u + \sum_{\jb\in \{\etwo, \cdots, \eDbar\}} \int_0^t \langle \ell\, |_{\jb}, \Wsig_u \rangle\, \d W_u^{(\jb)}\,.
    \end{align}
    Thus, if $\ell\in \Nn$, then $\ell^{\emptyword}=0$, $\ell\, |_{\eone}  + \frac{1}{2} \sum_{\jb\in \{\etwo, \cdots, \eDbar\}} \ell\, |_{\jb\, \jb}\in \Nn$, and $\ell\, |_{\jb}\in \Nn$ for all $\jb\in \{\etwo, \cdots, \eDbar\}$. This implies that $\ell\, |_{\jb}\in \Nn$ for all $\jb\in \{\eone, \etwo, \cdots, \eDbar\}$, which in turn implies that $\ell\, |_{\vb} \in \Nn$ for any word $\vb$. Let $\vb$ be the longest word such that $\ell^\vb \neq 0$. Then, we would have $\ell\, |_{\vb} \in \Nn$; however, $\ell\, |_{\vb} = \ell^\vb\, \emptyword$, implying $\ell^\vb = 0$, a contradiction. Thus, $\Nn = \{ \emptyword\}$. 
\end{proof}
Using Lemma \ref{lemma: injectiveness of signature control}, we can show that the tensor representation of the cost functional inherits the strict convexity property from the original cost functional.
\begin{lemma}
\label{lemma:criterion_convexity}
   The cost functions $\langle \mfq(\cdot), \E[\Wsig_T] \rangle$ and $\langle \mfq^\lev(\cdot), \E[\Wsig_T] \rangle$ are strictly convex in $\left(T((\R^\Dbar)^*)\right)^K$.
\end{lemma}
\begin{proof}
    The convexity of $\langle \mfq(\cdot), \E[\Wsig_T] \rangle$ follows from the convexity of the actual cost function and Lemma \ref{lemma: injectiveness of signature control}, so we only need to show the convexity of $\langle \mfq^\lev(\cdot), \E[\Wsig_T] \rangle$. Let us consider any $\ctt, \tilde \ctt\in \left(T((\R^\Dbar)^*)\right)^K$ such that $\ctt \neq \tilde \ctt$. Define $\ct:=\la \ctt, \Wsig\ra$ and $\tilde \ct:=\la \tilde \ctt, \Wsig\ra$. From Lemma \ref{lemma: injectiveness of signature control}, we have $\ct^{(k)} \neq \tilde \ct^{(k)}$. Denote by $\st$ and $\tilde \st$ the controlled systems associated with $\ct$ and $\tilde \ct$ respectively and take $\stt$, $\tilde{\stt}$ as their tensor respective representation, i.e.  $\st = \langle \stt, \Wsig \rangle$ and  $\tilde \st = \langle \tilde{\stt}, \Wsig \rangle$. Let us take $s\in [0,1]$ and define $$\hat \ctt := s\, \ctt + (1-s)\, \tilde \ctt, \quad \hat \ct := s\, \ct + (1-s)\,\tilde \ct\,.$$ Then $\hat \ct^{(k)} = \la \hat \ctt^{(k)}, \Wsig\ra$ for all $k=1,\cdots,K$. Denote by $\hat \st$ the controlled system associated with $\hat \ct$ and $\hat{\stt}$ as its tensor representation, i.e. $\hat \st = \langle \hat{\stt}, \Wsig \rangle$. It follows that $$\hat \st = s\, \st + (1-s)\, \tilde \st \quad \text{and}\quad \hat{\stt} = s\, \stt + (1-s)\, \tilde{\stt} \,.$$ By taking a truncation up to level $\lev$, we have $$\hat{\stt}^{\lev} = s\, \stt^{\lev} + (1-s)\, \tilde{\stt}^{\lev}\,. $$ From the linear-quadratic form of the performance criterion, we have
    \begin{align}
        \langle \mfq^\lev(\hat \ctt), \E[\Wsig_T] \rangle \leq s\, \langle \mfq^\lev(\ctt), \E[\Wsig_T] \rangle + (1-s)\, \langle \mfq^\lev(\tilde \ctt), \E[\Wsig_T] \rangle \,.
    \end{align}
    Moreover, as $B(t)$ is continuous and strictly positive definite, the inequality above is strict if $s\in(0,1)$. Thus, $\langle \mfq^\lev(\cdot), \E[\Wsig_T] \rangle$ is strictly convex.
\end{proof}

\subsection{Convergence of the truncated problem and density results}

In this section, we will prove that minimising the cost functional over $\mcA_\Sig$ is the same as minimising it over the larger space $\mcA$. Moreover, we are going to propose a truncation method to approximate the infinite-dimensional optimisation problem by a finite-dimensional one while preserving convexity and consistency. The following theorem sums up the main findings in this section.

\begin{theorem}
\label{theorem: main result}
    For sufficiently large truncation levels $\lev,M$, the minimization of the truncated cost functional $\langle \mfq^\lev(\ctt), \E[\Wsig_T] \rangle$ over $\ctt \in \big(T^M((\R^\Dbar)^*)\big)^K$ is a proxy for the original linear quadratic optimization problem in the sense that
    \begin{equation}
    \label{eq: main result 1}
        \lim_{M\to \infty} \lim_{\lev \to \infty} \inf_{\ctt \in \big(T^M((\R^\Dbar)^*)\big)^K} \langle \mfq^\lev(\ctt), \E[\Wsig_T] \rangle = \inf_{\ct \in \mcA} J(\ct)\,,
    \end{equation}
    and
    \begin{equation}
    \label{eq: main result 2}
        \lim_{M\to \infty} \lim_{\lev \to \infty} \langle \ctt^{\lev,M},\Wsig\rangle = u^* \text{ weakly in } L^2(\d \Pb \otimes \d t)
    \end{equation}
    where $\ctt^{\lev,M}$ denotes the minimizer of the truncated problem and $u^*$ is the optimal control in the linear quadratic optimization problem.
\end{theorem}

To prove Theorem~\ref{theorem: main result}, first we need the following result from \cite{ceylan2025universality} on the $L^2$-universality of linear functionals acting on Brownian signatures.
\begin{theorem}
\label{theorem: density result}
    For every $\ct \in \mcA$, there exists $(\ctt_m^{(k)})_{m=1}^\infty \subset T((\R^\Dbar)^*)$, $k=1,..., K$, such that
    \begin{equation}
        \lim_{m \to \infty} \max_{k=1,\dots, K} \E \left [ \int_0^T \big|\ct_t^{(k)} - \langle \ctt_m^{(k)}, \Wsig_t \rangle\big|^2\, \d t\right] = 0\,.
    \end{equation}
\end{theorem}
\begin{proof}
    This follows directly from (\cite{ceylan2025universality}, Corollary 4.3.(i)). In fact, the density result still holds in the space of $L^p$-integrable processes for $p>1$.

    \,
\end{proof}
Employing Theorem \ref{theorem: density result}, we have the following corollary.
\begin{corollary}
\label{prop: infimum over signature controls}
    We have that $\inf_{\ct \in \mcA} J(\ct) = \inf_{\ct \in \mcA_\Sig} J(\ct)\,.$
\end{corollary}
\begin{proof}
    Let us choose any $\ct\in \mcA$. Then we can choose $(\ctt_m^{(k)})_{m=1}^\infty \subset \mcA_\Sig$ satisfying the statement of Theorem \ref{theorem: density result} and define $\ct_m$ by $\ct_{m, t}^{(k)}:= \langle \ctt_m^{(k)}, \Wsig_t\rangle$. Due to the linear drift and volatility of $\st$, we obtain
    \begin{equation}
        \lim_{m\to \infty} \E \left[ \sup_{0\leq s\leq T} |\st_s^\ct - \st_s^{\ct_m}|^2\right] = 0\,,\footnote{This can be proven using standard techniques involving Cauchy-Schwarz, Itô isometry, Burkholder-Davis-Gundy, and Gronwall inequalities.}
    \end{equation}
    and thus from the linear-quadratic nature of the cost functional, we have $J(\ct_m) \to J(\ct)$. Thus, the set of real numbers $\{ J(\ct)\, |\, \ct \in \mcA_\Sig \}$ is dense in $\{ J(\ct)\, |\, \ct \in \mcA \}$, and the equality follows.
\end{proof}

Now, we are ready to prove our main result.

\begin{proof}[Proof of Theorem \ref{theorem: main result}]
    Let us fix $M\in \N$. From Proposition \ref{prop: cost functional as linear functional}, Lemma \ref{lemma: full cost functional as polynomial} and \ref{lemma:criterion_convexity}, if one restricts oneself to the space $\big(T^M((\R^\Dbar)^*)\big)^K$, the coefficients  of the (strictly convex) quadratic polynomial $\langle \mfq^\lev(\cdot), \Wsig_T \rangle$  converge for $\lev\to \infty$ to those of the strictly convex quadratic polynomial $\langle \mfq(\cdot), \E[\Wsig_T] \rangle$. The same is then true also for the respective infima, i.e.
    $$ \lim_{\lev \to \infty} \inf_{\ctt \in \big(T^M((\R^\Dbar)^*)\big)^K} \langle \mfq^\lev(\ctt), \E[\Wsig_T] \rangle = \inf_{\ctt \in \big(T^M((\R^\Dbar)^*)\big)^K} \langle \mfq(\ctt), \E[\Wsig_T] \rangle\,.$$
    
    By taking $M\to \infty$, applying Proposition \ref{prop: cost functional as linear functional} and Corollary \ref{prop: infimum over signature controls}, we obtain
    
    \begin{align}
         \lim_{M \to \infty} &\lim_{\lev \to \infty} \inf_{\ctt \in \big(T^M((\R^\Dbar)^*)\big)^K} \langle \mfq^\lev(\ctt), \E[\Wsig_T] \rangle 
         = \lim_{M \to \infty} \inf_{\ctt \in \big(T^M((\R^\Dbar)^*)\big)^K} \langle \mfq(\ctt), \E[\Wsig_T] \rangle\\
         &= \inf_{\ctt \in \big(T((\R^\Dbar)^*)\big)^K} \langle \mfq(\ctt), \E[\Wsig_T] \rangle
         = \inf_{\ct \in \mcA_\Sig} J(\ct)
         = \inf_{\ct \in \mcA} J(\ct)\,.
    \end{align}
    In fact, when $L \to \infty$, it follows that even the minimiser $\ctt^{\lev, M}$ of $\langle \mfq^\lev(\cdot), \Wsig_T \rangle$ over $\big(T^M((\R^\Dbar)^*)\big)^K$  converges to $\ctt^M$, the minimiser of $\langle \mfq(\cdot), \Wsig_T \rangle$ over this domain.

    Let us argue why $\ct^M:=\langle \ctt^M,\Wsig\rangle$ converges weakly to $\ct^*$, the minimiser of $J$ over $\mcA$. From the above, we deduce that $J(\ct^M)\to \inf_{\ct \in \mcA} J(\ct)$. In particular, $(\ct^M)_{M=1,2,\dots}$ is bounded in $L^2(\Pb \otimes dt)$. Thus, by the Banach-Alaoglu theorem, there exists a weak accumulation point $\tilde{\ct}^*$ and, by Mazur's theorem, we can find $\tilde{\ct}^M \in \mathrm{conv}(\ct^M,\ct^{M+1},\dots)$, $M=1,2,\dots$, converging to $\tilde{\ct}^*$ in the strong $L^2(\Pb \otimes dt)$-sense. By standard stability results of linear SDEs, the cost functional is continuous with respect to the latter convergence and so
    \begin{align}
        J(\tilde{\ct}^*)=\lim_{M} J(\tilde{\ct}^M) \leq \lim_M J(\ct^M)=\inf_{\ct \in \mcA} J(\ct),
    \end{align}
     where the estimate is due to the convexity of $J$. It follows that $\tilde{\ct}^*$ minimises $J$ over~$\mcA$. By uniqueness of this minimiser,  $\ct^M$ thus converges weakly in $L^2$ to the minimiser of $J$ over $\mcA$.
\end{proof}
   
\subsection{Beyond Brownian noise}

Observe that the extended tensors $\stt^\ctt$ (associated with the controlled state) and $\mfq(\ctt)$ (associated with the cost function) do not depend on the driving noise $W$. As such, heuristically, we may replace the driving noise in controlled linear systems \eqref{eq:multidim_system} with a general process $Y$, such as a fractional Brownian motion, with signature $\Ysig$ and hope that a similar result to Theorem \ref{theorem: main result} would still hold, i.e.
\begin{equation}
    \lim_{M\to \infty} \lim_{\lev \to \infty} \inf_{\ctt \in \big(T^M((\R^\Dbar)^*)\big)^K} \langle \mfq^\lev(\ctt), \E[\Ysig_T] \rangle \overset{?}{=} \inf_{\ct \in \mcA} J(\ct)\,.
\end{equation}
Here, there are two main challenges to establish this broader result:
\begin{enumerate}
    \item It is now unclear whether the process $\langle \stt^\ctt, \Ysig\rangle$ is even well-defined; we need the truncated pairing $\langle \big(\stt^\ctt\big)^{\leq \lev}, \Ysig\rangle$ to converge to the process $\st^\ct$ in some sense, and the latter process also needs to have a proper meaning. But our result (Theorem \ref{theorem: main result}) to this effect makes essential use of the Brownian motion and its signature.
    \item There is currently limited $L^p$-universality result for linear functionals of \textit{classical} signatures beyond Brownian motion. Such a result does exist for \textit{robust} signatures $\lambda(\Ysig)$ (see \cite{bayer2025primal}); however, we are not aware of a way to find a tensor $\mfq(\ctt)$ such that $\langle \mfq(\ctt), \lambda(\Ysig) \rangle=J(\ct)$ for $\ct=\langle \ctt, \E[\lambda(\Ysig)]\rangle$. Recently, there is a new $L^p$-universality result by \cite{chevyrev2026orthogonal} for linear functionals of a class of rough paths which includes Gaussian and Markovian rough paths.
\end{enumerate}

\section{Numerical experiments}
\label{sec:numerics}

In what follows, we study the one-dimensional case ($N=D=K=1)$ with two types of noise: (i) Brownian motion, and (ii) fractional Brownian motion with Hurst index $H=1/4$. We also consider state truncation levels $\lev$ up to $5$ and control truncation levels $M$ up to $3$; as we will see, this is enough for our problem. To work with signatures, we employ  the ``iisignature'' package by \cite{iisignature} and source codes from \cite{arribas2020optimal, cartea2022double}. Additionally, we utilise the ``fbm'' package to simulate fractional Brownian motion and the ``differint'' package by \cite{adams2019differint} to perform fractional calculus operations.\footnote{ The code is publicly  available in \href{https://github.com/muhammadalifaqsha/optimal-control-with-signatures}{muhammadalifaqsha}'s GitHub.}

\subsection{Brownian noise}

In the Brownian case, we set  $\st_0 = 10$, $A(t) = C(t) = D(t) = G =0$, and the remaining model parameters of the state dynamics and cost functional to be one. By Fawcett's formula (see \cite{fawcett2003problems, lyons2004cubature}), we utilise the following closed-form equation for the expected Brownian signature,
\begin{align}
    \E [\Wsig_T] = \sum_{n=0}^\infty \frac{T^n}{n!}\, \left(\eone + \frac{1}{2}\,  \etwo \otimes \etwo \right)^{\otimes n}\,.
\end{align}
Using the above formula, we are able to calculate all coefficients of the quadratic polynomial explicitly for the truncated problem. For $\ctt\in T^M((\R^\Dbar)^*)$, let us write
\begin{align}
    \langle \mfq^\lev(\ctt), \E [\Wsig_T]\rangle = \left(\sum_{\substack{\vb_1 \in V_{m_1}, \vb_2 \in V_{m_2}\\ m_1 + m_2\leq M}} c^{\vb_1, \vb_2}\, \ctt^{\vb_1}\, \ctt^{\vb_2}\right) + \left(\sum_{\substack{\vb \in V_m\\ m\leq M}} c^\vb\, \ctt^{\vb}\right) + c_0.
\end{align}
Suppose that we have a black-box function that evaluates $\langle \mfq^\lev(\cdot), \E [\Wsig_T]\rangle$. It is easy to check that
\begin{align}
    c_0 &= \langle \mfq^\lev(\emptyword), \E [\Wsig_T]\rangle\,, \quad
    c^{\vb, \vb} = \frac{\langle \mfq^\lev(\vb), \E [\Wsig_T]\rangle + \langle \mfq^\lev(-\vb), \E [\Wsig_T]\rangle}{2} - c_0\,,\\
    c^{\vb} &= \frac{\langle \mfq^\lev(\vb), \E [\Wsig_T]\rangle - \langle \mfq^\lev(-\vb), \E [\Wsig_T]\rangle}{2}\,,\\
    c^{\vb_1, \vb_2} &= \frac{\langle \mfq^\lev(\vb_1 + \vb_2), \E [\Wsig_T]\rangle - c^{\vb_1, \vb_1} - c^{\vb_2, \vb_2} - c_0}{2}\,\quad \forall \vb_1\neq \vb_2\,.
\end{align}

After obtaining the polynomial coefficients, we apply first-order condition to minimise $\langle \mfq^\lev(\cdot), \E [ \Wsig_T]\rangle$, thus obtaining the coefficients of the minimizer $\ctt^{\lev, M}$. Afterwards, we simulate $\ct^{\lev, M} = \langle \ctt^{\lev, M}, \Wsig\rangle$ a total of $20,000$ times by partitioning the time horizon into $1,000$ equally-spaced intervals. As benchmark, we compute the actual optimal control (obtained by solving the HJB equation) and simulate the performance using  the same randomness as in the signature approach. The performance of the strategies is reported in the first and second panel of Figure \ref{fig:performance_diff}.

\begin{figure}[ht]
    \centering
    \includegraphics[width=0.49\linewidth]{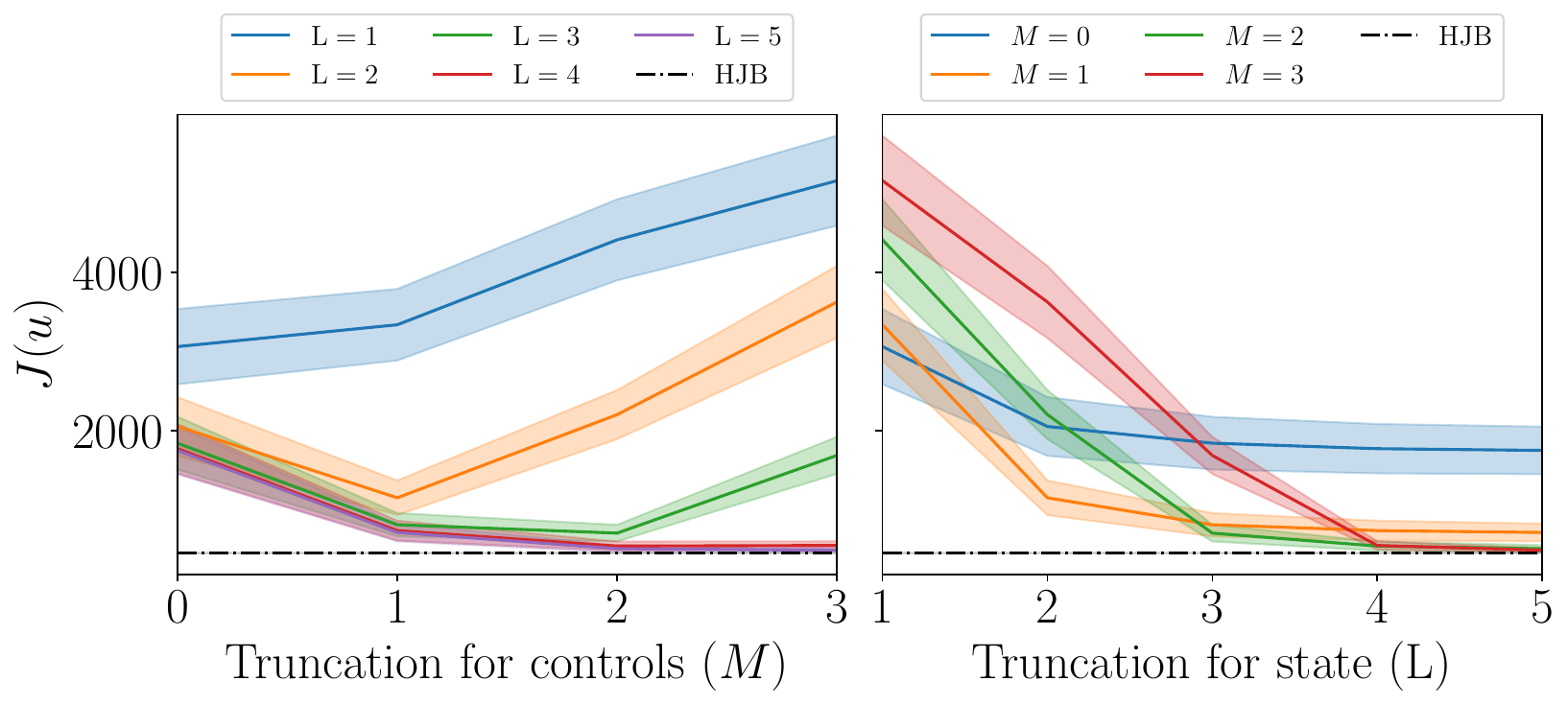}
    \includegraphics[width=0.49\linewidth]{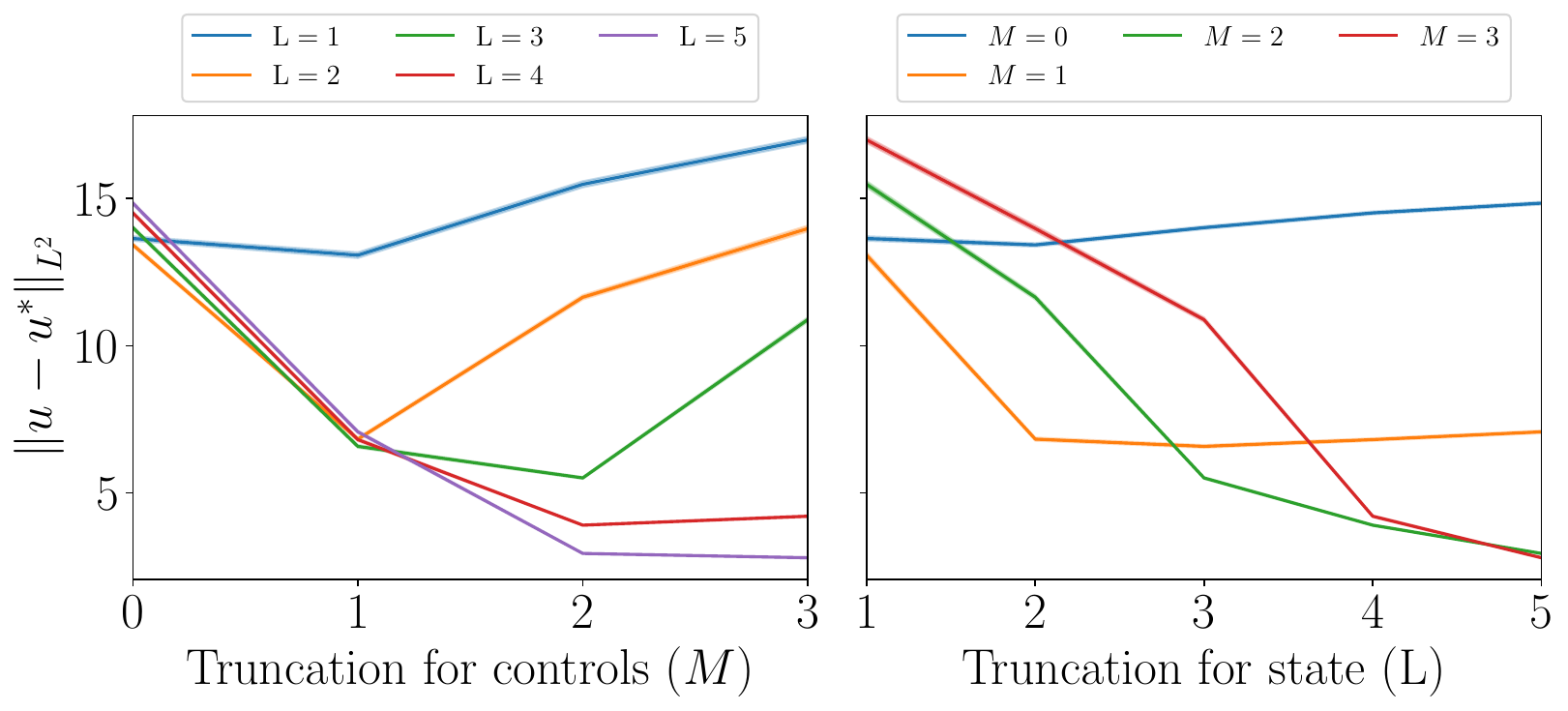}
    \caption{First and second panel: Estimation of $J(\ct)$ with 95\% confidence interval (in this case $J(u^*) = 455$). Third and fourth panel: Estimation of $L^2(\d \Pb \otimes \d t)$ distance between $u^{\lev, M}$ and the actual optimal control $u^*$ with 95\% confidence interval. }
    \label{fig:performance_diff}
\end{figure}

On the first and third panel of Figure \ref{fig:performance_diff}  we observe that the level of state truncation ($\lev$)  strongly influences the optimization results. If one were to choose a low $\lev$, solving the truncated problem would produce a highly suboptimal cost, no matter how complex the control model (more complexity = higher truncation $M$ for the control). In fact, for a given $\lev$, the cost value appears to decrease initially and then increase after around $M= \lev-1$. This is due to the fact that in Equation \eqref{eq:tensor_equation_3}, the coefficients of $\ctt$ at a level $m$ appear only in the coefficients of $\stt^{\ctt}$ at levels $m+1$ or higher. As the coefficients of $\ctt$ at higher levels will not affect the coefficients of $\stt^\ctt$ at lower levels, the procedure can be considered as an overfit, i.e. optimising a complex control model with less accurate target function. This is also reflected in our main result (Theorem \ref{theorem: density result}) that we have an ordered limit where $\lev$ should be significantly larger than $M$. 

On the second and fourth panel of Figure \ref{fig:performance_diff}   we observe that for each level of control truncation ($M$), the cost value decreases as we increase the accuracy of the cost functional (more accuracy corresponds to higher level of state truncation $\lev$). The cost values seem to plateau to quantities above the actual optimal cost, but said quantities seem to converge to the optimal cost as the complexity of the controls ($M$) increases.

\subsection{Fractional Brownian motion ($H=1/4$)}

In this case, we set $\st_0=b_0=\sigma_0=C(t)=D(t)=E=G=0$, and the remaining model parameters of the state dynamics and cost functional to be one. The optimal control thus seeks to minimize the expected average squared control effort and distance from the origin of a geometric fractional Brownian motion whose drift is additively controlled.

First, we implement Monte-Carlo simulation ($10,000$ sample path with $1,000$ time partitions) to approximate the expected fractional Brownian signature $\E [\Wsig^H_t]$. Next, we apply the first-order condition to minimise $\langle \mfq^\lev(\cdot), \hat{\E} [\Wsig_T^H]\rangle$ as in the Brownian case to obtain the (approximate) minimizer $\ctt^{\lev, M}$. Afterwards, we simulate $\ct^{\lev, M} = \langle \ctt^{\lev, M}, \Wsig^H\rangle$ a total of  $10,000$ times by partitioning the time horizon into $1,000$ equally-spaced intervals. As benchmark, we compute the optimal control from Corollary 3.4 in \cite{duncan2013linear} and simulate the performance using  the same randomness as in the signature approach. Note that this benchmark is also an approximate to the actual optimal control as we have to approximate several fractional integrals. The performance of the strategies is reported in Figure \ref{fig:performance_diff_fbm}.

\begin{figure}[H]
    \centering
    \includegraphics[width=0.49\linewidth]{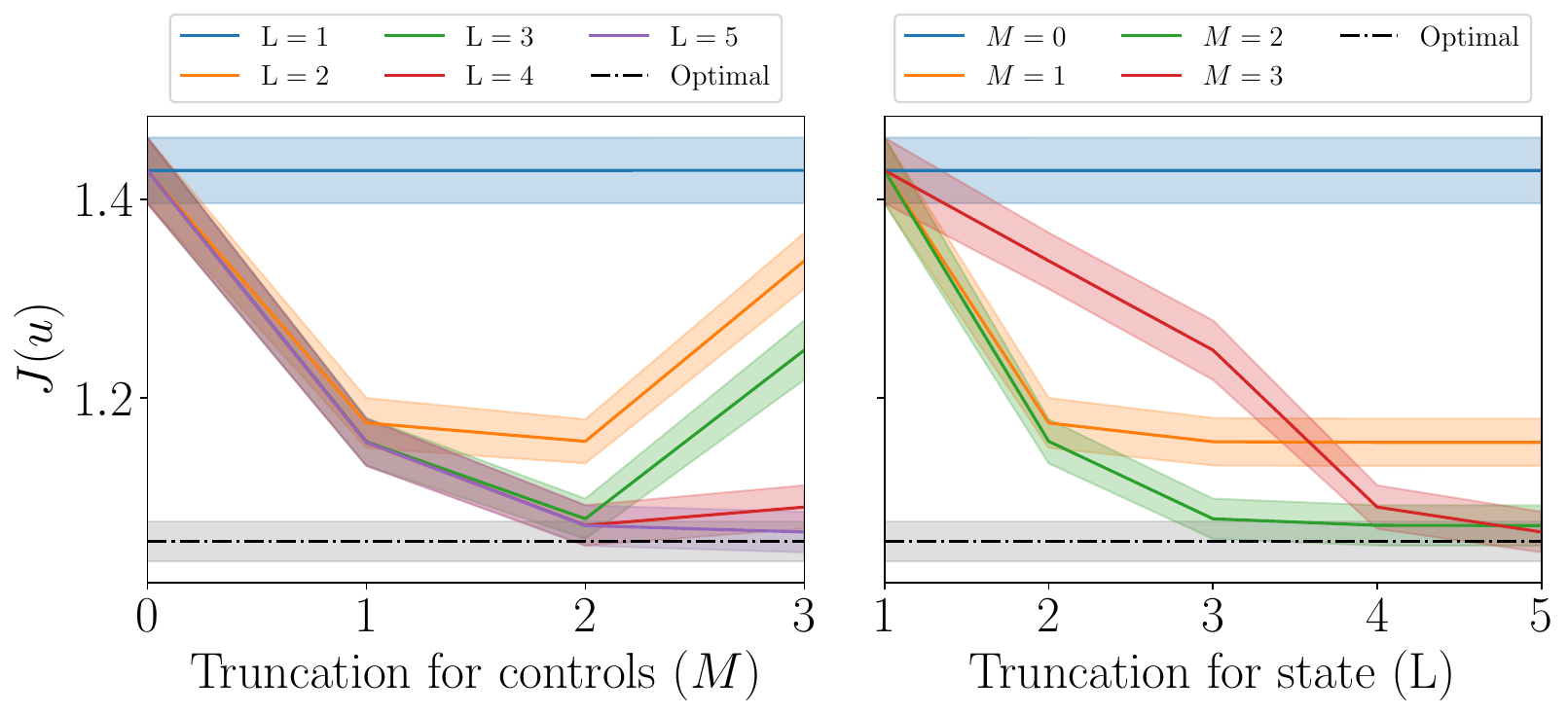}
    \includegraphics[width=0.49\linewidth]{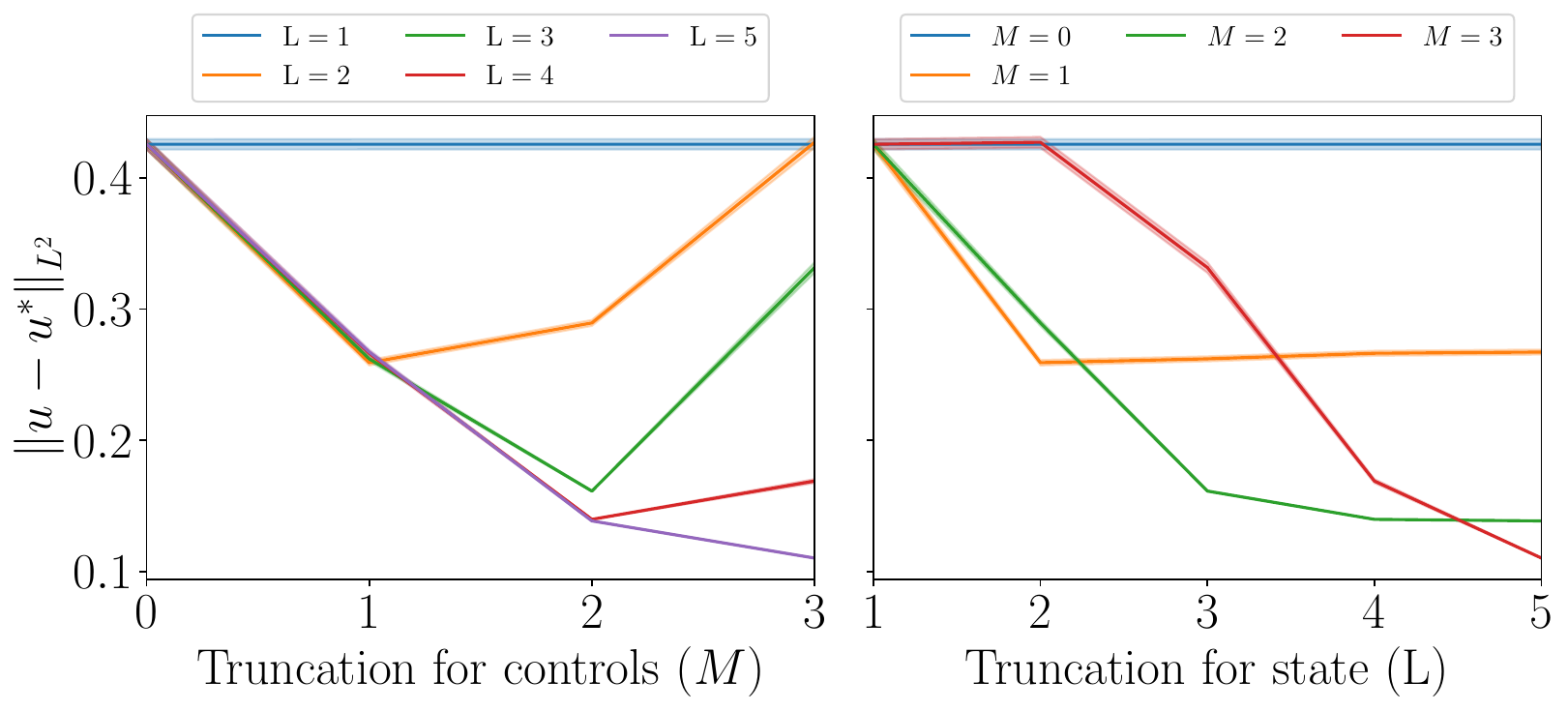}
    \caption{First and second panel: Estimation of $J(\ct)$ with 95\% confidence interval ($J(u^*) \approx 1.056$). Third and fourth panel: Estimation of $L^2(\d \Pb \otimes \d t)$ distance between $u^{\lev, M}$ and the benchmark $u^*$ with 95\% confidence interval. }
    \label{fig:performance_diff_fbm}
\end{figure}

The above results show that we still obtain similar convergence trend as in the Brownian case. Showing that although the theoretical convergence guarantees are yet to be established, the numerical experiments produce encouraging results.

\bibliographystyle{siamplain}
\bibliography{references}

\end{document}

%% file: ex_shared.tex
\usepackage{lipsum}
\usepackage{amsmath, amsfonts, amssymb}
\usepackage{graphicx}
\usepackage{epstopdf}
\usepackage{mathtools}
\usepackage{autonum}
\usepackage{algorithmic}
\usepackage{multirow}
\usepackage{doi}
\ifpdf
  \DeclareGraphicsExtensions{.eps,.pdf,.png,.jpg}
\else
  \DeclareGraphicsExtensions{.eps}
\fi

\usepackage{tikz-cd}
\usepackage{hyperref}
\hypersetup{
	colorlinks = true, 
	urlcolor = blue, 
	linkcolor = red, 
	citecolor = blue 
}

\newsiamremark{remark}{Remark}
\newsiamremark{hypothesis}{Hypothesis}
\crefname{hypothesis}{Hypothesis}{Hypotheses}
\newsiamthm{claim}{Claim}
\newsiamthm{assumption}{Assumption}
\newsiamremark{fact}{Fact}
\newsiamremark{illustration}{Illustration}
\crefname{fact}{Fact}{Facts}

\DeclareMathSymbol{\shortminus}{\mathbin}{AMSa}{"39}
\usepackage{shuffle}
\def\shpr{\shuffle}

\newcommand{\la}{\left \langle}
\newcommand{\ra}{\right\rangle}

\newcommand{\N}{\mathbb{N}}
\newcommand{\R}{\mathbb{R}}
\newcommand{\E}{\mathbb{E}}
\newcommand{\Pb}{\mathbb{P}}

\newcommand{\lev}{{\mathrm{L}}}
\newcommand{\mfT}{{\mathfrak{T}}}

\newcommand{\mcA}{{\mathcal{A}}}

\newcommand{\mcF}{{\mathcal{F}}}

\newcommand{\mbF}{{\mathbb{F}}}

\newcommand{\Wsig}{{\widehat{\mathbb{W}}}}

\newcommand{\Ysig}{{\widehat{\mathbb{Y}}}}

\newcommand{\wb}{{{\color{blue} \mathbf{w}}}}
\newcommand{\vb}{{\color{blue} \mathbf{v}}}

\newcommand{\Sig}{{\mathrm{Sig}}}

\newcommand{\ext}{{\textup{ext}}}
\newcommand{\tensorLspace}{{\mcT^{\Dbar}}}

\newcommand{\tensorExpspace}{{\mcT^{\exp}}}

\newcommand{\F}{\mathcal{F}}

\newcommand{\mfq}{\mathtt{J}}

\newcommand{\mcW}{\mathcal{W}}

\newcommand{\Proj}{\mathrm{Proj}}

\newcommand{\Nn}{\mathcal{N}}

\newcommand{\emptyword}{{\color{blue} \pmb{\phi}}}
\newcommand{\eone}{{\color{blue} \mathbf{1}}}
\newcommand{\etwo}{{\color{blue} \mathbf{2}}}

\newcommand{\eD}{{\color{blue} {\mathbf{D}}}}
\newcommand{\eDbar}{{\color{blue} \bar{\mathbf{D}}}}
\newcommand{\mcT}{\mathcal{T}}
\newcommand{\jb}{{\color{blue} \mathbf{j}}}

\newcommand{\Dbar}{{\bar{D}}}
\newcommand{\alphabet}{{\color{blue} \mathbf{a}}}

\renewcommand{\d}{\mathrm{d}}

\newcommand{\ct}{u} 
\newcommand{\ctt}{\mathtt{\ct}}
\newcommand{\st}{X} 
\newcommand{\stt}{\mathtt{\st}}

\usepackage{amsopn}

\usepackage{enumitem}


%% file: references.bib
@book{carmona2016lectures,
  title={Lectures on BSDEs, stochastic control, and stochastic differential games with financial applications},
  author={Carmona, Ren{\'e}},
  year={2016},
  publisher={SIAM}
}

@book{yong1999stochastic,
  title={Stochastic controls: Hamiltonian systems and HJB equations},
  author={Yong, Jiongmin and Zhou, Xun Yu},
  volume={43},
  year={1999},
  publisher={Springer Science \& Business Media}
}

@inproceedings{arribas2020sigsde,
    author = {Arribas, Imanol Perez and Salvi, Cristopher and Szpruch, Lukasz},
    title = {Sig-SDEs model for quantitative finance},
    year = {2021},
    isbn = {9781450375849},
    publisher = {Association for Computing Machinery},
    address = {New York, NY, USA},
    booktitle = {Proceedings of the First ACM International Conference on AI in Finance},
    articleno = {7},
    numpages = {8},
    location = {New York, New York},
    series = {ICAIF '20}
}

@article{sirignano2018dgm,
  title={DGM: A deep learning algorithm for solving partial differential equations},
  author={Sirignano, Justin and Spiliopoulos, Konstantinos},
  journal={Journal of computational physics},
  volume={375},
  pages={1339--1364},
  year={2018},
  publisher={Elsevier}
}

@inproceedings{hoglund2023neuralrde,
  title={A neural RDE approach for continuous-time non-Markovian stochastic control problems},
  author={H{\"o}glund, Melker and Ferrucci, Emilio and Hern{\'a}ndez, Camilo and Gonzalez, Aitor Muguruza and Salvi, Cristopher and S{\'a}nchez-Betancourt, Leandro and Zhang, Yufei},
  year={2023},
  booktitle={ICML Workshop on New Frontiers in Learning, Control, and Dynamical Systems}
}

@article{jaber2024pathdependent,
  title={Path-dependent processes from signatures},
  author={Abi Jaber, Eduardo  and G{\'e}rard, Louis-Amand and Huang, Yuxing},
  journal={arXiv preprint arXiv:2407.04956},
  year={2024}
}

@article{lyons2020non,
  title={Non-parametric pricing and hedging of exotic derivatives},
  author={Lyons, Terry and Nejad, Sina and Perez Arribas, Imanol},
  journal={Applied Mathematical Finance},
  volume={27},
  number={6},
  pages={457--494},
  year={2020},
  publisher={Taylor \& Francis}
}

@article{iisignature,
    title={Algorithm 1004: The iisignature Library:
    Efficient Calculation of Iterated-Integral Signatures and Log Signatures},
    author={Jeremy Reizenstein and Benjamin Graham}, journal={ACM Transactions on Mathematical Software (TOMS)}, year={2020}
}

@phdthesis{ni2012expected,
    title={The expected signature of a stochastic process},
    author={Ni, H.},
    year={2012},
    school={University of Oxford}
}

@book{bensoussan2018estimation,
  title={Estimation and Control of Dynamical Systems},
  author={Alain Bensoussan},
  year={2018},
  publisher={Springer Cham}
}

@article{jaber2024signature,
  title={Signature volatility models: pricing and hedging with {F}ourier},
  author={Abi Jaber, Eduardo and G{\'e}rard, Louis-Amand},
  journal={SIAM Journal on Financial Mathematics},
  volume={16},
  number={2},
  pages={606--642},
  year={2025},
  publisher={SIAM}
}

@article{arribas2020optimal,
  title={Optimal execution with rough path signatures},
  author={Kalsi, Jasdeep and Lyons, Terry and Arribas, Imanol Perez},
  journal={SIAM Journal on Financial Mathematics},
  volume={11},
  number={2},
  pages={470--493},
  year={2020},
  publisher={SIAM}
}

@article{ceylan2025universality,
  title={Global universal approximation with Brownian signatures},
  author={Ceylan, Mihriban and Pr{\"o}mel, David J},
  journal={arXiv preprint arXiv:2512.16396},
  year={2025}
}

@InProceedings{platen1980approx,
    author="Platen, E. and Wagner, W.",
    editor="Grigelionis, Bronius",
    title="Approximation of {I}t{\^o} integral equations",
    booktitle="Stochastic Differential Systems Filtering and Control",
    year="1980",
    publisher="Springer Berlin Heidelberg",
    address="Berlin, Heidelberg",
    pages="172--176"
}

@article{lyons2004cubature,
 author = {Terry Lyons and Nicolas Victoir},
 journal = {Proceedings: Mathematical, Physical and Engineering Sciences},
 number = {2041},
 pages = {169--198},
 publisher = {The Royal Society},
 title = {Cubature on Wiener Space},
 volume = {460},
 year = {2004}
}

@article{chevyrev2022signature,
  title={Signature moments to characterize laws of stochastic processes},
  author={Chevyrev, Ilya and Oberhauser, Harald},
  journal={Journal of Machine Learning Research},
  volume={23},
  number={176},
  pages={1--42},
  year={2022}
}

@article{cuchiero2025global,
  title={Global universal approximation of functional input maps on weighted spaces},
  author={Cuchiero, Christa and Schmocker, Philipp and Teichmann, Josef},
  journal={Constructive Approximation},
  pages={1--76},
  year={2026},
  publisher={Springer}
}

@article{platen1982taylor,
    author="Platen, E. and Wagner, W.",
    title="On a Taylor formula for a class of {I}t\^o processes",
    journal="Probability and Mathematical Statistics",
    year="1982",
    volume="3",
    pages="37--51"
}

@phdthesis{fawcett2003problems,
  title={Problems in Stochastic Analysis: Connections Between Rough Paths and Non-commutative Harmonic Analysis},
  author={Fawcett, T.},
  year={2003},
  school={University of Oxford}
}

@Inbook{kloeden1992stochastic,
    author="Kloeden, Peter E.
    and Platen, Eckhard",
    title="Stochastic Taylor Expansions",
    bookTitle="Numerical Solution of Stochastic Differential Equations",
    year="1992",
    publisher="Springer Berlin Heidelberg",
    address="Berlin, Heidelberg",
    pages="161--226"
}

@article{chen1954iterated,
    title = {Iterated Integrals and Exponential Homomorphisms},
    author = {Kuo-Tsai Chen},
    journal = {Proceedings of the London Mathematical Society},
    volume = {s3-4},
    number = {1},
    pages = {502-512},
    year = {1954}
}

@article{boedihardjo2016uniqueness,
title = {The signature of a rough path: Uniqueness},
journal = {Advances in Mathematics},
volume = {293},
pages = {720-737},
year = {2016},
author = {Horatio Boedihardjo and Xi Geng and Terry Lyons and Danyu Yang}
}

@article{hambly2010uniqueness,
  title={Uniqueness for the signature of a path of bounded variation and the reduced path group},
  author={Hambly, Ben and Lyons, Terry},
  journal={Annals of Mathematics},
  pages={109--167},
  year={2010},
  publisher={JSTOR}
}

@article{bank2025stochastic,
  title={Stochastic control with signatures},
  author={Bank, Peter and Bayer, Christian and Hager, Paul P and Riedel, Sebastian and Nauen, Tobias},
  journal={SIAM Journal on Control and Optimization},
  volume={63},
  number={5},
  pages={3189--3218},
  year={2025},
  publisher={SIAM}
}

@article{futter2025kernel,
  title={Kernel Learning for Mean-Variance Trading Strategies},
  author={Futter, Owen and Cirone, Nicola Mu{\c{c}}a and Horvath, Blanka},
  journal={arXiv preprint arXiv:2507.10701},
  year={2025}
}

@article{bayer2025primal,
  title={Primal and dual optimal stopping with signatures},
  author={Bayer, Christian and Pelizzari, Luca and Schoenmakers, John},
  journal={Finance and Stochastics},
  pages={1--34},
  year={2025},
  publisher={Springer}
}

@article{bayer2023optimal,
  title={Optimal stopping with signatures},
  author={Bayer, Christian and Hager, Paul P and Riedel, Sebastian and Schoenmakers, John},
  journal={The Annals of Applied Probability},
  volume={33},
  number={1},
  pages={238--273},
  year={2023},
  publisher={Institute of Mathematical Statistics}
}

@article{jaber2025hedging,
  title={Hedging with memory: shallow and deep learning with signatures},
  author={Abi Jaber, Eduardo  and G{\'e}rard, Louis-Amand},
  journal={arXiv preprint arXiv:2508.02759},
  year={2025}
}

@article{futter2023signature,
  title={Signature trading: A path-dependent extension of the mean-variance framework with exogenous signals},
  author={Futter, Owen and Horvath, Blanka and Wiese, Magnus},
  journal={arXiv preprint arXiv:2308.15135},
  year={2023}
}

@article{jaber2025signatureapproach,
  title={Signature approach for pricing and hedging path-dependent options with frictions},
  author={Abi Jaber, Eduardo  and Hainaut, Donatien and Motte, Edouard},
  journal={arXiv preprint arXiv:2511.23295},
  year={2025}
}

@book{lyons2007differential, 
  title={Differential equations driven by rough paths},
  author={Lyons, Terry and Caruana, Michael and L{\'e}vy, Thierry},
  year={2007},
  publisher={Springer}
}

@article{cartea2022double,
  title={Double-execution strategies using path signatures},
  author={Cartea, {\'A}lvaro and Arribas, Imanol P{\'e}rez and S{\'a}nchez-Betancourt, Leandro},
  journal={SIAM Journal on Financial Mathematics},
  volume={13},
  number={4},
  pages={1379--1417},
  year={2022},
  publisher={SIAM}
}

@article{adams2019differint,
      title={differint: A Python Package for Numerical Fractional Calculus}, 
      author={Matthew Adams},
      year={2019},
      journal={arXiv preprint arXiv:1912.05303}
}

@article{duncan2013linear,
    author = {Duncan, Tyrone E. and Pasik-Duncan, Bozenna},
    title = {Linear-Quadratic Fractional Gaussian Control},
    journal = {SIAM Journal on Control and Optimization},
    volume = {51},
    number = {6},
    pages = {4504-4519},
    year = {2013}
}

@article{duncan2017stochastic,
    title = {Stochastic Linear-Quadratic Control with State Dependent Fractional Brownian Noise and Stochastic Coefficients},
    journal = {IFAC-PapersOnLine},
    volume = {50},
    number = {2},
    pages = {199-202},
    year = {2017},
    note = {Control Conference Africa CCA 2017},
    author = {Duncan, Tyrone E. and Pasik-Duncan, Bozenna}
}

@article{kleptsyna2003lqregulator,
    author = {Kleptsyna, M. L. and Breton, Alain Le and Viot, M.},
    title = {About the linear-quadratic regulator problem under a fractional brownian perturbation},
    journal = {ESAIM: Probability and Statistics},
    pages = {161--170},
    year = {2003},
    publisher = {EDP-Sciences},
    volume = {7}
}

@Inbook{chevyrev2026aprimer,
    author="Chevyrev, Ilya
    and Kormilitzin, Andrey",
    title="A Primer on the Signature Method in Machine Learning",
    bookTitle="Signature Methods in Finance: An Introduction with Computational Applications",
    year="2026",
    publisher="Springer Nature Switzerland",
    address="Cham",
    pages="3--64"
}

@article{gaines1994thealgebra,
    author = {J. G. Gaines},
    title = {The algebra of iterated stochastic integrals},
    journal = {Stochastics and Stochastic Reports},
    volume = {49},
    number = {3-4},
    pages = {169--179},
    year = {1994},
    publisher = {Taylor \& Francis}
}

@article{cass2024lecturenotesroughpaths,
      title={Lecture notes on rough paths and applications to machine learning}, 
      author={Thomas Cass and Cristopher Salvi},
      year={2024},
      journal={arXiv preprint arXiv:2404.06583}
}

@article{chevyrev2026orthogonal,
      title={Orthogonal polynomials on path-space}, 
      author={Ilya Chevyrev and Emilio Ferrucci and Darrick Lee and Terry Lyons and Harald Oberhauser and Nikolas Tapia},
      year={2026},
      journal={arXiv preprint arXiv:2602.18808} 
}

@book{pham2009continuous,
  title={Continuous-time stochastic control and optimization with financial applications},
  author={Pham, Huy{\^e}n},
  volume={61},
  year={2009},
  publisher={Springer Science \& Business Media}
}
